\documentclass[11pt,reqno]{amsart} 
\usepackage[margin=1in]{geometry} 
\usepackage{amsmath,amsthm,amsfonts,amssymb,amscd,amsrefs}
\usepackage{enumitem}

\usepackage{times}
\usepackage{esint,stackrel}
\usepackage{enumitem}

\usepackage{xfrac}

\usepackage{cases}

\usepackage{fancyhdr} 
\def\p{\partial} 
\def\eps{\varepsilon}

\newcommand*{\sgn}{\ensuremath{\mathrm{sgn\,}}}
\newcommand{\RR}{\mathbb R}

\newcommand{\TT}{\mathbb T}
\newcommand{\OO}{\mathcal O}

\renewcommand*{\bar}{\overline}

\renewcommand{\epsilon}{\varepsilon}

\newtheorem{theorem}{Theorem}[section]
\newtheorem{lemma}[theorem]{Lemma}
\newtheorem{proposition}[theorem]{Proposition}

\numberwithin{equation}{section}

\newcommand{\be}{\begin{equation}}
\newcommand{\ee}{\end{equation}}

\newcommand{\R}{\mathbb{R}}

\def\cir{ \! \circ \! }

\newcommand{\loc}{\text{loc}}

\newcommand{\W}{\mathring{W}}
\newcommand{\Z}{\mathring{Z}}
\newcommand{\C}{\mathring{\Sigma}}
\newcommand{\K}{\mathring{K}}

\usepackage[colorlinks=true, pdfstartview=FitV, linkcolor=blue,citecolor=blue, urlcolor=blue]{hyperref}

\title{A New Type of Stable Shock Formation in Gas Dynamics}

\author{Isaac Neal}
\address{Courant Institute of Mathematical Sciences, New York University, New York, NY 10012.}
\email{\href{in577@cims.nyu.edu}{in577@cims.nyu.edu}}

\author{Calum Rickard}
\address{Department of Mathematics, University of California Davis, Davis, CA 95616.}
\email{\href{crickard@math.ucdavis.edu}{crickard@math.ucdavis.edu}}

\author{Steve  Shkoller}
\address{Department of Mathematics, University of California Davis, Davis, CA 95616.}
\email{\href{shkoller@math.ucdavis.edu}{shkoller@math.ucdavis.edu}}

\author{Vlad Vicol}
\address{Courant Institute of Mathematical Sciences, New York University, New York, NY 10012.}
\email{\href{vicol@cims.nyu.edu}{vicol@cims.nyu.edu}}

\begin{document}

\begin{abstract}
From an open set of initial data,  we construct a family of classical solutions to the 1D nonisentropic compressible Euler equations which form $C^{0,\nu}$ cusps as a first singularity, for any $\nu \in [\frac{1}{2}, 1)$. For this range of $ \nu $, this is the first result demonstrating the stable formation of such $C^{0,\nu} $ cusp-type singularities, also known as {\it pre-shocks}. The proof  uses a new formulation of the differentiated Euler equations along the fast acoustic characteristic, and relies  on  a novel set of  $L^p$ energy estimates for all $1<p<\infty$, which may be of independent interest.
\end{abstract}

\maketitle

\allowdisplaybreaks
 
\section{Introduction}

A core line of inquiry in the study of nonlinear partial differential equations is the analysis of the finite-time breakdown of classical solutions. For systems of conservation laws like the compressible Euler equations, the prototypical form of finite-time breakdown is a {\it{shock}}, where the smooth solution stays continuous but steepens to a cusp at a point in spacetime known as the {\it{pre-shock}}, after which time it can no longer exist as a smooth solution. In recent decades it has become apparent that having access to the precise behavior of the solution at the time of the first singularity (the {\it{shock profile}}), is key to determining how the resulting shock waves and weak singularities develop and propagate. Here we study the shock formation problem for the Euler equations of one-dimensional ideal gases, and we provide a detailed description of a family of {\em stable} shocks which  resemble $-\sgn(x)|x|^\nu$ near the pre-shock for $\nu \in [\frac{1}{2}, 1)$. 

Here and throughout the paper, by the stability of the shock formation process we mean that for arbitrary sufficiently small and smooth perturbations (with respect to a suitable topology) of the initial data,  the same type of pre-shock forms, possibly at a different location in space and time (which may be computed exactly),  and with the same $C^{0,\nu}$ shock profile  (up to dilations).

\subsection{The Euler equations}

The compressible Euler equations in one space dimension are given by
\begin{subequations} 
\label{euler-weak}
\begin{align}
\partial_t (\rho u)  + \p_y (  \rho u^2 + p  ) & =0\,,  \label{ee2}\\
\partial_t \rho + \p_y (  \rho u ) &=0\,, \label{ee1}\\
\partial_t E + \p_y ( (p+ E) u ) &=0 \label{ee3}\, ,
\end{align}
\end{subequations}
where the unknowns $u, \rho, E,$ and $p$ are scalar functions defined on $\RR \times \RR$ or $\TT \times \RR$: $u$ is the fluid velocity, $\rho$ is the (everywhere positive) fluid density, $E$ is the specific total energy, and $p$ is the pressure. To close the system, one must introduce an equation of state that relates the internal energy $E-\tfrac{1}{2}\rho u^2$ to $p$ and $\rho$. If the fluid is an {\it{ideal gas}}, the equation of state is
\begin{equation}
\label{eq:ig}
p  = (\gamma-1)(E-\tfrac{1}{2}\rho u^2),
\end{equation}
where $\gamma >1$ is a fixed constant called the {\it{adiabatic exponent}}. 

For ideal gases, the specific entropy, $S$, satisfies the relation
\begin{equation}
\label{idealgas}
p = \tfrac{1}{\gamma} \rho^\gamma e^S.
\end{equation}
If we introduce the parameter $\alpha : = \tfrac{\gamma-1}{2}$ and define the (renormalized) sound speed\footnote{ The actual sound speed in a compressible fluid is $c =\sqrt{\tfrac{\p p}{\p \rho} } $.}  to be 
\begin{equation}
\sigma = \tfrac{1}{\alpha} \sqrt{\tfrac{\p p}{\p \rho} } = \tfrac{1}{\alpha} e^{S/2} \rho^\alpha,
\end{equation}
then we can rewrite the ideal gas equations \eqref{euler-weak}--\eqref{eq:ig} in terms of $u,\sigma,$ and $S$ as follows:
\begin{subequations}
\label{eq:euler:2}
\begin{align}
\p_tu + u \p_y u+ \alpha \sigma \p_y \sigma & = \tfrac{\alpha}{2\gamma} \sigma^2 \p_y S,   \\
\p_t \sigma +  u \p_y\sigma + \alpha \sigma \p_y  u & = 0,   \\
\p_t S + u \p_y S & = 0.
\end{align}
\end{subequations}

In this paper, we prove the following theorem:

\begin{theorem}[\bf Main Result, abbreviated]
\label{thm:abbreviated}
For any choice of adiabatic exponent $\gamma > 1$,  and parameter $\beta \geq 1$, there is an open set (a ball of radius $\eps \ll 1$ with respect to the $W^{2,\frac{\beta}{\beta-1}}(\mathbb{T})$ topology\footnote{See the statement of Theorem~\ref{thm:main}, which in fact gives a weaker assumption for the dominant Riemann variable.}) of initial data $(u_0, \sigma_0, S_0)$,  for which the unique local-in-time classical solution $(u,\sigma, S)$ of the Cauchy problem for \eqref{eq:euler:2} forms a gradient blowup singularity at a computable time $T_* \approx \frac{2}{1+\alpha}$. Furthermore, at time $T_*$ the solution has the profile
\begin{align*}
u(y,T_*) & = u(y_*,T_*) - \Bigl((1+\tfrac{1}{\beta})^{\frac{\beta}{\beta+1}}+ \OO(\eps) \Bigr) \sgn(y-y_*) |y-y_*|^{\frac{\beta}{\beta+1}} + \bigl( 1+ \OO(\eps) \bigr) (y-y_*) \\
\sigma(y,T_*) & = \sigma(y_*,T_*) - \Bigl((1+\tfrac{1}{\beta})^{\frac{\beta}{\beta+1}}+ \OO(\eps) \Bigr) \sgn(y-y_*) |y-y_*|^{\frac{\beta}{\beta+1}} + \bigl( 1+ \OO(\eps) \bigr) (y-y_*) \\
S(y,T_*) & = S(y_*,T_*) + \p_y S(y_*,T_*)(y-y_*) + \OO\Bigl( |y-y_*|^{\frac{\beta+2}{\beta+1}}\Bigr)
\end{align*}
for $y$ in a neighborhood of diameter $\approx 1$ about a point $y_*$ which is computable from the initial data. 
\end{theorem}

The above theorem provides the first constructive proof of stable  $C^{\frac{\beta}{\beta+1}}$-cusp type singularity formation for $\beta \ge 1$.   
We refer to \S~\ref{sec:ID} and Theorem~\ref{thm:main} below for the precise assumptions on the initial data, and the precise statement of our main result.

\subsection{Prior results}
\label{sec:prior}

It has been known for some time that solutions of  the Euler equations and more general systems of hyperbolic PDE can 
form singularities in finite time from smooth initial data, see the seminal papers~\cite{lax1964development,John74,liu1979development,majda1984compressible,sideris1985formation}. These classical results were non-constructive and asserted a breakdown of smoothness\footnote{These arguments are based on an ODE comparison-principle (which necessitates a priori assumptions on the boundedness of the density and the continuity of the gradient of the  solution) with a Riccati equation which blows up in finite time, thus yielding a proof by contradiction.}  without a description of the actual mechanism of singularity formation. See~\cite{chen2002cauchy,dafermos2005hyperbolic,liu2021shock} for further references.

Recently, a number of constructive {\it shock formation} results have 
been established~\cite{lebaud1994description,chen2001formation,Kong2002,yin2004formation,christodoulou2007formation,luk2018shock,BuShVi2019a,BuShVi2019b,BuShVi2020,BuDrShVi2021,luk2021stability,buckmaster2022formation,neal2023}. These constructive results show that for an appropriate class of smooth initial data the {\em first} gradient singularity is of shock-type and no other singularity can occur prior. The solution at this first singularity, i.e.~the {\em pre-shock}, only possesses limited H\"older regularity, which describes the type of cusp that forms.  To date, only the $C^{\frac{1}{3}}$  H\"{o}lder class pre-shock (corresponding to $\beta = \tfrac{1}{2}$ in the language of Theorem~\ref{thm:abbreviated}) has been extensively studied~\cite{lebaud1994description,chen2001formation,Kong2002,BuShVi2019a,BuShVi2019b,BuShVi2020,BuDrShVi2021,neal2023}, as this $C^{\frac{1}{3}}$ cusp singularity emerges in a stable fashion from a large open set of smooth and {\it generic} initial data. 

To the best of our knowledge, for the compressible Euler dynamics there are no results establishing the formation of $C^{\frac{\beta}{\beta+1}}$-cusp type pre-shocks with $\beta > \frac{1}{2}$, corresponding to H\"older exponents {\em strictly larger} than $\frac 13$. 
In fact, the intuition based on the study of the Burgers equation\footnote{For $\beta > 0$, a family of globally self-similar solutions to the 1D Burgers equation  with a $C^{\frac{\beta}{\beta+1}}$ profile have been known for some time~\cite{eggers2008role, collot2018singularityburgers}.} only suggests the possible emergence of a $C^\frac{1}{2k+1}$-cusp, corresponding to $\beta = \frac{1}{2k} $ with  $k\ge 1$ an integer. Moreover, one expects in these particular cases that the shock formation process is generically unstable, but has finite-codimension stability. This intuition was successfully implemented in the context of the 2D Euler equations with azimuthal symmetry to establish the finite time formation of a $C^{\frac{1}{5}}$ pre-shock (corresponding to $\beta=\frac 14$) in \cite{buckmaster2022formation}. The  finite-codimension stable formation of $C^\frac{1}{2k+1}$-cusp type pre-shocks (corresponding to  $\beta = \frac{1}{2k}$) for all integers $k\geq 2$ was established recently  in~\cite{IyRiShVi2023}.

\subsection{New Ideas}
\label{sec:ideas}

This paper makes use of the   {\it{differentiated Riemann variables}}  \eqref{def:q^w} first introduced in \cite{BuDrShVi2021} and later utilized in \cite{neal2023} to study the generic $C^ {\frac{1}{3}} $ pre-shock formation from $C^4$ initial data.   The
smoothness of the data employed in \cite{BuDrShVi2021,neal2023}  permitted the use of  a pointwise charicteristic-based approach.  On the other hand, the initial data required for the formation  of $C^{\frac{\beta}{\beta+1}}$ pre-shocks with
$\beta \ge1$ have limited regularity and, in particular, such data is not even $C^2$ (see \S~\ref{sec:ID}).  
As such, we have devised a novel  approach to this analysis which consists of the following elements:
\begin{itemize}
\item[1.]  Both the dominant and subdominant differentiated Riemann variables, together with the entropy, are studied 
along the {\it fast acoustic characterisic} (rather than the actual characteristic families which propagate each 
disturbance).  This allows us to {\it effectively freeze}  both the dominant Riemann variable and the entropy in this flow,
reducing the required bounds to only the subdominant Riemann variable (and its derivatives).

\item[2.]  $L^p$-based energy estimates are then obtained for this system of variables, assuming limited regularity
initial data in the Sobolev space $W^{2,p}$.  This energy method allows for a spatially-global analysis for functions
whose second derivatives are not continuous, and produces bounds which are uniform in $p$ and time, and hence
allow us to pass to the limit as $p \to \infty $.

\item[3.] We show that there exists an open set of initial data in the $W^{2,p}$ topology for which 
 $C^{\frac{\beta}{\beta+1}}$ pre-shock formation occurs for $\beta\ge1$.  This establishes that this new class of cusp-type
 first singularities is stable under perturbation of the initial data by small disturbances of class 
  $W^{2,\frac{\beta}{\beta-1}}(\mathbb{T})$ (or smoother).

\end{itemize}

\subsection{Riemann variables}
\label{sec:RV}

The Riemann variables $w$ and $z$ are defined by
\begin{align}\label{riemann}
w : = u + \sigma \,,
\qquad \mbox{and} \qquad 
z := u-\sigma.
\end{align}
If we further adopt the notation $k : = S$, then the Euler equations become
\begin{align}
\label{eq:euler:RV}
\p_t w + \lambda_3 \p_y w & = \tfrac{\alpha}{2\gamma} \sigma^2 \p_y k, \notag \\
\p_t z + \lambda_1 \p_y z & = \tfrac{\alpha}{2\gamma} \sigma^2 \p_y k, \notag \\
\p_t k + \lambda_2 \p_y k & = 0,
\end{align}
where
\begin{align*}
\lambda_1 & : = u-\alpha \sigma = \tfrac{1-\alpha}{2} w + \tfrac{1+\alpha}{2} z, \\
\lambda_2 & : = u = \tfrac{1}{2} w + \tfrac{1}{2} z, \\
\lambda_3 & : = u+ \alpha \sigma = \tfrac{1+\alpha}{2} w + \tfrac{1-\alpha}{2} z.
\end{align*}
Define $\psi, \phi,$ and $\eta$ to be the flow of $\lambda_1, \lambda_2,$ and $\lambda_3$ respectively.

\section{Motivation: Burgers' Equation}
\label{sec:motivation}

Our motivation comes from a simple solution to Burgers' equation,
\begin{equation}
\label{eq:burgers}
w_t + ww_y = 0,
\end{equation}
which is a special case of \eqref{eq:euler:RV} where $\alpha =1$, $z_0 \equiv 0$, and $k_0$ is constant. In addition to being a special case of the 1D isentropic Euler equations expressed in Riemann variables, Burgers' equation is the archetypal equation for shock formation and development.

If $w$ is a solution of \eqref{eq:burgers} and $\eta$ is the flow of $w$, i.e. $\eta$ solves
$$
\tfrac{d}{dt} \eta(x,t) = w(\eta(x,t),t),
\qquad
\eta(x,0) = x,
$$
then $w$ must remain constant along the flow lines of $\eta$. From this it follows that if $w(x,0) = w_0(x)$ then
\begin{equation}
\eta(x,t) = x+ tw_0(x).
\end{equation} 
When $w_0$ is $C^1$ and $w_0'$ attains its global minimum at $x_*$ the solution of Burgers' equation remains $C^1$ smooth up to time
$$ T_* = \begin{cases} -\frac{1}{w'_0(x_*)} \hspace{10mm} w'_0(x_*) < 0 \\ +\infty \hspace{15mm} w'_0(x_*) \geq 0 \end{cases}, $$
and in the case where $w'_0(x_*) < 0$ we have $\eta_x(x_*,T_*) = 0$ and $\p_y w(\eta(x_*,T_*), T_*) = -\infty$.

\subsection{Key example}
\label{sec:example}

Let $\beta > 0$ and let $w$ be the solution of Burgers' equation \eqref{eq:burgers} on $\RR$ with initial data
\begin{equation*}
w_0(x) = -x + Cx|x|^{\frac{1}{\beta}}
\end{equation*}
for some $C > 0$. We compute that
$$ w'_0(x) = -1 + (1+\tfrac{1}{\beta}) C |x|^{\frac{1}{\beta}}, $$
so $w'_0$ has a unique global minimum at $x=0$, and it blows up at time $T_* =1$. At the blowup time, we have
$$ \eta(x,1) = Cx|x|^{\frac{1}{\beta}}. $$
Therefore, if $y = \eta(x,1)$ then
$$ x = \sgn(y) \big(\tfrac{|y|}{C}\big)^{\frac{\beta}{\beta+1}}, $$
and
$$ w(y,1) = w_0(x) = - \sgn(y) C^{-\frac{\beta}{\beta+1}} |y|^{\frac{\beta}{\beta+1}} + y. $$

\subsection{Stability for $\beta \geq 1$}
\label{sec:stable:burgers}

When $\beta \geq 1$, the geometry of the shock profile $w(\cdot, T_*)$ obtained in \S~\ref{sec:example} is stable under perturbations in $C^{1, \frac{1}{\beta}}$. Let $\mathcal E \in C^{1,\frac{1}{\beta}}_{\loc}(\RR)$ be such that
\begin{itemize}
\item $\mathcal E'(0) < 1$, and
\item $[\mathcal E']_{C^{0,\frac{1}{\beta}}} < (1+ \frac{1}{\beta}) C$,
\end{itemize}
and let
$$ w_0(x) = -x + Cx|x|^{\frac{1}{\beta}} + \mathcal E(x). $$
Then
\begin{align*}
w'_0(x) & = -1 + \mathcal E'(0) + (1+ \tfrac{1}{\beta}) C|x|^{\frac{1}{\beta}} + (\mathcal E'(x)-\mathcal E'(0)) \\
& \geq -1 + \mathcal E'(0) + \big[ (1+ \tfrac{1}{\beta}) C - [\mathcal E']_{C^{0,\frac{1}{\beta}}} \big] |x|^{\frac{1}{\beta}}.
\end{align*}
So $w'_0$ still has a unique global minimum at $x = 0$. Since $\mathcal E'(0) < 1$, the minimum of $w'_0$ is negative and we have
$$ T_* = \frac{1}{1-\mathcal E'(0)}. $$
For all $x$ we have
$$ \mathcal E(x) = \mathcal E(0) + \mathcal E'(0) x + h(x)x|x|^{\frac{1}{\beta}} $$
where
$$ h(x) = \frac{\int^x_0 \mathcal E'(x') - \mathcal E'(0) \: dx'}{x|x|^{\frac{1}{\beta}}}. $$
It is immediate that
$$ |h(x)| \leq \tfrac{\beta}{1+\beta}  [\mathcal E']_{C^{0,\frac{1}{\beta}}}  < C $$
everywhere. If we let $y = \eta(x,T_*)$, then we have
\begin{align*}
y 
= x + T_*w_0(x) 
= T_* \mathcal E(0) + T_*(C + h(x)) x|x|^{\frac{1}{\beta}}  
=: y_* +a(x) x|x|^{\frac{1}{\beta}}.
\end{align*}
Since $|h(x)| < C$ everywhere, we have $a(x) > 0$ everywhere, so $\sgn(y-y_*) = \sgn(x)$ and
$$ x = \sgn(y-y_*) \big\vert \frac{y-y_*}{a(x)}\big\vert^{\frac{\beta}{\beta+1}}. $$
It now follows that
\begin{align*}
w(y,T_*)  
= w_0(x)
& = -x + Cx|x|^{\frac{1}{\beta}} + \mathcal E(x) \\
& = \mathcal E(0) + (\mathcal E'(0) -1) x + (C + h(x))\sgn(x) |x|^{1+\frac{1}{\beta}} \\
& = \mathcal E(0) - \tfrac{(1-\mathcal E'(0))}{|a(x)|^{\frac{\beta}{\beta+1}}} \sgn(y-y_*) |y-y_*|^{\frac{\beta}{\beta+1}} + (C+h(x)) \sgn(y-y_*) \big\vert \frac{y-y_*}{a(x)}\big\vert \\
& = \mathcal E(0) - \tfrac{(1-\mathcal E'(0))}{|a(x)|^{\frac{\beta}{\beta+1}}} \sgn(y-y_*) |y-y_*|^{\frac{\beta}{\beta+1}} + (1-\mathcal E'(0)) (y-y_*).
\end{align*}

We have proven the following: 
\begin{proposition}[\bf Motivating Example: Stable Shock Formation for Burgers]
\label{prop:burgers}
If $\beta \geq 1$, $C > 0$, and
$$ w_0 = -x + Cx|x|^{\frac{1}{\beta}}+ \mathcal E(x), $$
where $\mathcal E \in C^{1,\frac{1}{\beta}}_{\loc}(\R)$ with $\mathcal E'(0) < 1$ and $[\mathcal E']_{C^{0,\frac{1}{\beta}}} < (1+ \frac{1}{\beta}) C$ then Burgers' equation \eqref{eq:burgers} has a unique $C^1$ solution, $w$, with $w\big\vert_{t=0} = w_0$ on $\R\times [0,T_*)$, where $T_* = \tfrac{1}{1-\mathcal E'(0)}$. Moreover, $w$ is continuous up to time $T_*$ and at time $T_*$ we have
\begin{equation}
\label{eq:prop:burger}
w(y,T_*) = \mathcal E(0) -\sgn(y-y_*)b(y)|y-y_*|^{\frac{\beta}{\beta+1}} + (1-\mathcal E'(0)) (y-y_*)
\end{equation}
where $y_* : = \tfrac{\mathcal E(0)}{1-\mathcal E'(0)}$ and the function $b$ satisfies
\begin{align*}
\frac{((1- \mathcal E'(0))^{\frac{2\beta+1}{\beta+1}}}{\big( C + \frac{\beta}{\beta+1}  [\mathcal E']_{C^{0,\frac{1}{\beta}}} \big)^{\frac{\beta}{\beta+1}}} < b(y) < \frac{(1- \mathcal E'(0))^{\frac{2\beta+1}{\beta+1}}}{\big( C - \frac{\beta}{\beta+1}  [\mathcal E']_{C^{0,\frac{1}{\beta}}} \big)^{\frac{\beta}{\beta+1}}}
\end{align*}
for all $y \in \RR$.
\end{proposition}

In this paper, we will generalize Proposition \ref{prop:burgers} by letting the parameter $\alpha$ be any positive number (i.e. letting $\gamma$ be any number greater than $1$) and by letting $z_0$ and $k_0- \fint_\TT k_0$ be sufficiently small instead of identically zero. We will also work on the torus $\TT$ instead of $\RR$. See Theorem \ref{thm:main} below for details.

\subsection{Finite-Codimension Stability for $0< \beta < 1$, $\beta \neq \frac{1}{2}$}
\label{sec:unstable:burgers}

In the case where $0 < \beta < 1$ and $\beta \neq \frac{1}{2}$, the geometry of the shock profile obtained in  \S~\ref{sec:example}  is no longer stable under smooth perturbations. To see this, let $0< \beta <1$ and let the perturbation now be $\mathcal E(x) = \frac{\eps}{2}x^2 \varphi(x)$ where $\eps \neq 0$ is a small constant and $\varphi$ is a smooth bump function that is identically 1 for all $|x| \leq R$, for some $R\gg1$. Then one can check that $w'_0$ has a unique global minimum at the point
$$ x_* = - \left\vert \frac{\eps \beta^2}{(\beta+1) C} \right\vert^{\frac{\beta}{1-\beta}}, $$
and
$$ w'_0(x_*) = -1 -\frac{\eps^{\frac{1}{1-\beta}}}{((1+\frac{1}{\beta})C)^{\frac{\beta}{\beta-1}}} \bigg[ \beta^{-\frac{\beta}{1-\beta}}-\beta^{-\frac{1}{1-\beta}}\bigg] < 0. $$
So the solution $w$ still blows up in finite time, but since $x_* \neq 0$ one can now check that if $y = \eta(x,T_*)$ and $y_* = \eta(x_*,T_*)$ we have
$$ y-y_* \sim (x-x_*)^3 $$
in a sufficiently small neighborhood\footnote{ One can check that the implicit constants in the expression $y-y_* \sim (x-x_*)^3$ can be made $\sim |\eps|^{1-\frac{\beta}{1-\beta}}$, and therefore they blow up as $\eps \rightarrow 0$ when $\beta > \frac{1}{2}$,  remain close to 1 when $\beta = \frac{1}{2}$, and vanish as $\eps \rightarrow 0$ when $\beta < \frac{1}{2}$.} of $x_*$  and therefore
$$ w(y,T_*)-w(y_*,T_*) \sim (y-y_*)^{\frac{1}{3}} $$
for $|y-y_*|$ small enough. When $\beta = \frac{1}{2}$, this is still the same type of shock profile as before, but for all other choices of $\beta$ it is different. Perturbations of the form we have specified can be made arbitrarily small in all $C^N$ norms by sending $\eps \rightarrow 0$, so for $0< \beta<1, \beta \neq \frac{1}{2}$, the shock formation described in \S~\ref{sec:example}  is unstable.

However, when $0< \beta <1$, the example in \S~\ref{sec:example} is {\it{finite-codimension-stable}} with respect to suitable perturbations, as we shall show next. Let $1+\frac{1}{\beta} = n + \delta$, where $n \in \mathbb{Z}$ and $\delta \in (0,1]$. Since $\beta <1$, we have $n \geq 0$. Let $\mathcal E \in C^{n, \delta}_{\loc}(\RR)$ satisfy
\begin{itemize}
\item $ \mathcal E'(0) < 1$,
\item $\p^j_x\mathcal E(0) = 0$ for all $2 \leq j \leq n$,
\item $[\p^n_x \mathcal E]_{C^{0,\delta}} < C(1+\tfrac{1}{\beta})(1+\tfrac{1}{\beta}-1)\cdots (1+\tfrac{1}{\beta}-(n-1)), $
\end{itemize}
and choose initial data
$$ w_0(x) = -x + Cx|x|^{\frac{1}{\beta}} + \mathcal E(x). $$
We compute that
\begin{align*}
w'_0(x) & = -1 + \mathcal E'(0) + (1+ \tfrac{1}{\beta}) C|x|^{\frac{1}{\beta}} + (\mathcal E'(x)-\mathcal E'(0)) \\
& \geq -1 + \mathcal E'(0) + \bigg[ (1+ \tfrac{1}{\beta}) C - \tfrac{[\p^n_x\mathcal E]_{C^{0,\delta}}}{\tfrac{1}{\beta}(\tfrac{1}{\beta}-1) \cdots (\tfrac{1}{\beta}+2-n)} \bigg] |x|^{\frac{1}{\beta}}.
\end{align*}
So $w'_0$ has a unique global minimum at $x =0$. Since $\mathcal E'(0) < 1$, the minimum of $w'_0$ is negative and we have
$$ T_* = \frac{1}{1-\mathcal E'(0)} $$
as before. For all $x$ we have
$$ \mathcal E(x) = \mathcal E(0) + \mathcal E'(0) x + h(x)x|x|^{\frac{1}{\beta}} $$
where
$$ h(x) = \frac{\int^x_0 \int^{x_1}_0 \cdots \int^{x_{n-1}}_0 \p^n_x\mathcal E(x_n) \: dx_n dx_{n-1} \hdots dx_1}{x|x|^{\frac{1}{\beta}}}. $$
It is immediate that
$$ |h(x)| \leq \frac{[\p^n_x \mathcal E]_{C^{0,\delta}}}{(1+\tfrac{1}{\beta})(1+\frac{1}{\beta}-1) \cdots (1+\tfrac{1}{\beta}-(n-1))}  < C $$
for all $x \in \RR$. Therefore, if $y = \eta(x,T_*)$ then
\begin{align*}
y 
= x + T_*w_0(x)
= T_* \mathcal E(0) + T_*(C + h(x)) x|x|^{\frac{1}{\beta}} 
=: y_* +a(x) x|x|^{\frac{1}{\beta}}.
\end{align*}
Since $|h(x)| < C$ everywhere, we have $a(x) > 0$ everywhere, so $\sgn(y-y_*) = \sgn(x)$ and
$$ x = \sgn(y-y_*) \big\vert \frac{y-y_*}{a(x)}\big\vert^{\frac{\beta}{\beta+1}}. $$
It now follows that
\begin{align*}
w(y,T_*) 
= w_0(x)
& = -x + Cx|x|^{\frac{1}{\beta}} + \mathcal E(x) \\
& = \mathcal E(0) + (\mathcal E'(0) -1) x + (C + h(x))\sgn(x) |x|^{1+\frac{1}{\beta}} \\
& = \mathcal E(0) - \tfrac{(1-\mathcal E'(0))}{|a(x)|^{\frac{\beta}{\beta+1}}} \sgn(y-y_*) |y-y_*|^{\frac{\beta}{\beta+1}} + (C+h(x)) \sgn(y-y_*) \big\vert \frac{y-y_*}{a(x)}\big\vert \\
& = \mathcal E(0) - \tfrac{(1-\mathcal E'(0))}{|a(x)|^{\frac{\beta}{\beta+1}}} \sgn(y-y_*) |y-y_*|^{\frac{\beta}{\beta+1}} + (1-\mathcal E'(0)) (y-y_*).
\end{align*}

We have proven the following: 
\begin{proposition}[\bf Finite-Codimension Stable Shock Formation for Burgers]
\label{prop:burgers:unstable}
If $0< \beta <1$, $C > 0$, $1+ \frac{1}{\beta} = n + \delta$ with $n \in \mathbb{Z}, \delta \in (0,1]$, and
$$ w_0 = -x + Cx|x|^{\frac{1}{\beta}}+ \mathcal E(x), $$
where $\mathcal E \in C^{n,\delta}_{\loc}(\R)$ with $\mathcal E'(0) < 1$,
$$[\p^n_x \mathcal E]_{C^{0,\delta}} < C(1+\tfrac{1}{\beta})(1+\tfrac{1}{\beta}-1)\cdots (1+\tfrac{1}{\beta}-(n-1)), $$
and $\p^j_x\mathcal E(0) = 0$ for all $2 \leq j \leq n$, then Burgers' equation \eqref{eq:burgers} has a unique $C^1$ solution, $w$, with $w\big\vert_{t=0} = w_0$ on $\R\times [0,T_*)$, where $T_* = \tfrac{1}{1-\mathcal E'(0)}$. Moreover, $w$ is continuous up to time $T_*$ and at time $T_*$ we have
\begin{equation}
\label{eq:prop:burger:2}
w(y,T_*) = \mathcal E(0) -\sgn(y-y_*)b(y)|y-y_*|^{\frac{\beta}{\beta+1}} + (1-\mathcal E'(0)) (y-y_*)
\end{equation}
where $y_* : = \tfrac{\mathcal E(0)}{1-\mathcal E'(0)}$ and the function $b$ satisfies
\begin{align*}
\frac{((1- \mathcal E'(0))^{\frac{2\beta+1}{\beta+1}}}{\big( C + \tfrac{[\mathcal E']_{C^{0,\frac{1}{\beta}}}}{(1+\tfrac{1}{\beta})(1+\tfrac{1}{\beta}-1)\cdots (1+\tfrac{1}{\beta}-(n-1))}\big)^{\frac{\beta}{\beta+1}}} < b(y) < \frac{(1- \mathcal E'(0))^{\frac{2\beta+1}{\beta+1}}}{\big( C - \frac{[\mathcal E']_{C^{0,\frac{1}{\beta}}}}{(1+\tfrac{1}{\beta})(1+\tfrac{1}{\beta}-1)\cdots (1+\tfrac{1}{\beta}-(n-1))}\big)^{\frac{\beta}{\beta+1}}}
\end{align*}
for all $y \in \RR$.
\end{proposition}

We conjecture that similar finite-codimension stability results could be established for the 1D Euler, and we plan to address this in future work.

\subsection{Stablility for $\beta = \frac{1}{2}$}
In the case $\beta = \frac{1}{2}$, the shock from \S~\ref{sec:example} is stable (see e.g. \cite{neal2023}), but this case is of a different nature than the other choices of $\beta > 0$. In the proofs of both Proposition \ref{prop:burgers} and Proposition \ref{prop:burgers:unstable} we saw that the proof hinged on the blowup label $x_*$ where $w'_0$ has its global minimum coinciding with a place where $w_0$ satisfied
\begin{equation}
\label{eq:blowuplabel}
w_0(x) -w_0(x_*) - (x-x_*)w'_0(x_*) \sim (x-x_*) |x-x_*|^{\frac{1}{\beta}}.
\end{equation}
For a typical smooth function $w_0$  $x_*$ would be a nondegenerate critical point of $w'_0$, so we would expect 
$$  w_0(x) -w_0(x_*) - (x-x_*)w'_0(x_*) \sim (x-x_*)^3 $$
which is exactly what we want in the case $\beta = \frac{1}{2}$. So in \eqref{eq:blowuplabel} is expected when $\beta = \frac{1}{2}$. But, as we saw in \S~\ref{sec:unstable:burgers}, equation \eqref{eq:blowuplabel} being satisfied at the blowup label is not stable under smooth perturbations when $\beta < 1, \beta \neq \frac{1}{2}$.


\subsection{Remark: self-similar blowup}
\label{sec:selfsim}

A self-similar solution to Burgers' equation is a solution $w: \RR \times (-\infty, T) \rightarrow \RR$ of the form
\begin{equation}
\label{eq:selfsim}
w(y,t) = (T-t)^\beta W\big(\tfrac{y-y_0}{(T-t)^{\beta+1}}\big)
\end{equation}
where $W \in C^1_{\text{loc}}(\RR;\RR)$ is the {\it{similarity profile}}, $T \in \RR$ is the fixed blowup time, and $\beta > 0$ \footnote{We are ignoring the $\beta = 0$ case, where the self-similar profile is given by $W(x) = -x$. In this case, the profile doesn't satisfy the matching condition $W \sim |x|^{\frac{\beta}{\beta+1}}$ as $x \rightarrow \pm \infty$, and so it is different. See \cite{eggers2008role} for more.} is the {\it{scaling exponent}}. A function $w$ of the form \eqref{eq:selfsim} solves Burgers' equation if and only if the similarity profile solves the {\it{similarity equation}}
\begin{equation}
\label{eq:similarity}
-\beta W+ (1+\beta)x W_x + WW_x = 0.
\end{equation}

Consider the implicit equation
\begin{equation}
\label{eq:implicit}
x = -W-CW|W|^{\frac{1}{\beta}}
\end{equation}
where $C> 0$ is a constant.\footnote{ The previous literature up to this point has worked with the implicit equation $x=-W-C W^{1+\frac{1}{\beta}}$, see e.g.   \cite[ \S 11.2]{eggers2015singularities}, \cite[\S 2.4]{eggers2008role}, \cite{collot2018singularityburgers}.  However, this equation does not define a global solution for most values of $\beta$. For example, if $\beta =1$ then solution $W$ is not single-valued and there is no solution for $x > \frac{1}{4C}$. Our implicit equation remedies these issues, and gives the correct regularity which the similarity profile should have near $x=0$. E.g., when $\beta =1$ the similarity profile will be $C^{1,1}$ but not $C^2$.} Since $\frac{dx}{dW} = -1-C(1+\frac{1}{\beta})|W|^{\frac{1}{\beta}} \leq -1$ everywhere, it follows that $W \rightarrow x$ is a $C^1$ diffeomorphism of $\RR$ and therefore \eqref{eq:implicit} defines a $C^1$ similarity profile $W$ on all of  $\RR$ which is also strictly decreasing. In fact, $W$ is locally $C^{n,\delta}$, where $n \in \mathbb{Z}^+, \delta \in (0,1]$, and $1+\frac{1}{\beta} = n+\delta$. It is also easy to check that this profile $W$ solves the similarity equation \eqref{eq:similarity}.

Let $W$ be the solution to \eqref{eq:implicit}. To determine the growth of $W$ at infinity, we first note that $x = -W(1+C|W|^{\frac{1}{\beta}})$, so $\sgn(x) = -\sgn(W)$ everywhere. Therefore,
\begin{align*}
\frac{-\sgn(x)C^{-\frac{\beta}{\beta+1}}|x|^{\frac{\beta}{\beta+1}}}{W(x)} & = \frac{|x|^{\frac{\beta}{\beta+1}}}{|W|} 
= \bigg\vert \frac{W(1+C|W|^{\frac{1}{\beta}})}{|W|^{1+\frac{1}{\beta}}} \bigg\vert^{\frac{\beta}{\beta+1}} 
= \bigg\vert \frac{1}{|W|^{\frac{1}{\beta}}} + C \bigg\vert^{\frac{\beta}{\beta+1}}.
\end{align*}
Since $|W| \rightarrow \infty$ as $|x| \rightarrow \infty$, we have
\begin{equation}
\label{eq:asymptotics}
\frac{-\sgn(x)C^{-\frac{\beta}{\beta+1}}|x|^{\frac{\beta}{\beta+1}}}{W(x)}  \rightarrow 1 \qquad \text{ as } x \rightarrow \pm \infty.
\end{equation}
From \eqref{eq:asymptotics} it follows that if $w$ is defined by \eqref{eq:selfsim} with our choice of $W$ then
\begin{equation}
\label{eq:ssprofile}
\lim_{t \rightarrow T^-} w(y,t)  = -\sgn(y-y_0) C^{-\frac{\beta}{\beta+1}} |y-y_0|^{\frac{\beta}{\beta+1}}.
\end{equation}

Let us now examine the behavior of $W$ near $x =0$. Since $x = -W(1+C|W|^{\frac{1}{\beta}})$, we compute that in a neighborhood of $x =0$ we have
\begin{align*}
-x+Cx|x|^{\frac{1}{\beta}} & = W+CW|W|^{\frac{1}{\beta}}-CW|W|^{\frac{1}{\beta}} (1+C|W|^{\frac{1}{\beta}})^{\frac{1}{\beta}} - C^2 W |W|^{\frac{2}{\beta}} (1+ C|W|^{\frac{1}{\beta}})^{\frac{1}{\beta}} \\
& = W- (1+\tfrac{1}{\beta})C^2 W|W|^{\frac{2}{\beta}} + \OO \big( (1+\tfrac{1}{\beta}) C^3 |W|^{1+\frac{3}{\beta}} \big) \\
& = W + \OO\big( (1+\tfrac{1}{\beta})C^2 |x|^{1+\frac{2}{\beta}} \big).
\end{align*}
The last equation is true because $W(0) = 0$ and $W$ is $C^1$. Therefore, we conclude that
\begin{equation}
\label{eq:ssnearzero}
W = -x+Cx|x|^{\frac{1}{\beta}} + \OO\big( (1+\tfrac{1}{\beta})C^2 |x|^{1+\frac{2}{\beta}} \big)
\end{equation}
in a neighborhood of zero. Notice that the similarity profile agrees with the initial data chosen in \S~\ref{sec:example} near 0 if we drop the highest order term in this expansion.


\section{The result}

\subsection{Local well-posedness}
\label{sec:LWP}

It follows from the standard well-posedness theory of the Euler equations \eqref{eq:euler:2} that  if $s> \frac{3}{2}$ and  $w_0,z_0,k_0 \in H^s(\TT)$ then there exists a maximal time of existence $T_*  \in (0, +\infty]$, such that the following hold
\begin{itemize}
\item the equations \eqref{eq:euler:RV}  have a unique classical solution $(w,z,k) \in C^1(\TT \times [0,T_*); \RR^3)$ with initial data $(w_0, z_0, k_0)$;
\item the unique solution $(w,z,k)$ also satisfies
\begin{equation}
\label{lwp}
 (w,z,k) \in C([0,T_*) ; H^s(\TT)^3) \cap C^1([0,T_*) ; H^{s-1}(\TT)^3).
 \end{equation}
\end{itemize}
Furthermore, if $T_* < \infty$, then the following 
{\bf Eulerian blowup criterion} holds:
\begin{equation}
\label{eq:LCC}
\int^{T_*}_0 
\Bigl(\|\p_y w(\cdot, t)\|_{L^\infty} + \|\p_y z(\cdot, t)\|_{L^\infty} + \|\p_y k(\cdot, t)\|_{L^\infty}\Bigr)    dt = +\infty 
\,.
\end{equation}

Throughout the remainder of the paper, the initial data $w_0,z_0,k_0$ will always be assumed to lie in $W^{2,q}(\TT)$ for some $1< q \leq \infty$. In particular, the classical $H^s$ local-existence theory applies for a suitable value of $s>\frac 32$. To see this, recall that if $n \in \mathbb{Z}, n \geq 2$, and $1<q \leq \infty$, then
$$ W^{n,q}(\TT) \subset H^s(\TT) $$
with continuous embedding for
$$ s = s(n,q) = \begin{cases} n- (\tfrac{1}{q}-\tfrac{1}{2}) \hspace{18mm} p \leq 2 \\ n \hspace{35mm} p \geq 2 \end{cases}. $$
For all $n \geq 2$, $1< q \leq \infty$, we have $s > \frac{3}{2}$. So if $w_0,z_0, k_0 \in W^{2,q}(\TT)$ for some $1< q \leq \infty$, then $w_0,z_0,k_0 \in H^s(\TT)$ for some $s > \frac{3}{2}$ and our local well-posedness theory applies to the initial data $w_0,z_0, k_0$. As such, the {\it{blowup time}} $T_* \in (0,+\infty]$ will be the time at which the corresponding solution can no longer be continued as a solution in $C_tH^s_x \cap C^1_t H^{s-1}_x$ for any $s > \frac{3}{2}$.

\subsection{Assumptions on the initial data}
\label{sec:ID}

Let $\beta > 0$, and let $\bar{w}_0$ be a $2\pi$-periodic function such that
\begin{itemize}
\item $\bar{w}_0(x) = -x + \frac{\beta}{\beta+1} x|x|^{\frac{1}{\beta}}$ for $|x| \leq 1$, 
\item $|\bar{w}_0(x)| < 1$ for all $x \in \TT$,
\item $0 \leq \bar{w}'_0(x) \leq 1$ for $1\leq |x| \leq \pi$.
\end{itemize}
For our theorem we take $\beta \geq 1$, and the initial data $(w_0,z_0,k_0)$ will be such that $w_0$ is a small perturbation of $1+\bar{w}_0$ in $C^{1, \frac{1}{\beta}}$, $z_0$ is a small perturbation of $0$ in $W^{2,p}$, and $k_0$ is a small perturbation of a constant state in $W^{2,p}$, where $p = \frac{\beta}{\beta-1}$ is the H\"{o}lder conjugate of $\beta$.

We note that when $\beta > 1$ and $p= \frac{\beta}{\beta-1}$, $\bar{w}_0 \in C^{1,\frac{1}{\beta}}(\TT) \setminus W^{2,p}(\TT)$. Since $W^{2,p}(\TT)$ is not dense in $C^{1,\frac{1}{\beta}}$ for $\beta > 1$, small enough perturbations of $\bar{w}_0$ in $C^{1,\frac{1}{\beta}}$ will not be in $W^{2,p}$, so $w_0$ will not be in $W^{2,p}$. When $\beta = 1$ however, $C^{1,\frac{1}{\beta}} = C^{1,1} = W^{2,\infty}$, and we will be able to assume that $(w_0,z_0,k_0)$ lie in an open set in $W^{2,\infty}(\TT)$. However, $w_0$ will still not be $C^2$, because in order for $w_0$ to be sufficiently close to $1+\bar{w}_0$ in $W^{2,\infty}$, $w''_0$ it will have to have a jump discontinuity at 0.

\subsection{Notation}

At the beginning of most sections we will introduce a size parameter $\eps > 0$. Throughout the paper, the hypothesis that $\eps > 0$ was chosen to be sufficiently small will be implicit in the statements of all of our lemmas and propositions. How small $\eps > 0$ will need to be in order for our results to hold will depend on our choice of $\gamma$ (hence $\alpha$) and on $1-\|\bar{w}_0\|_{L^\infty}$, but doesn't depend on anything else. 

In what follows, we will often identify functions on the torus $\TT$ with $2\pi$-periodic functions on $\RR$ when the notation is most convenient. This identification will allow us to write expressions like ``$|x| \leq 1$". We will always use $x \in \TT$ to denote the Lagrangian label and $y \in \TT$ to denote the Eulerian variable.

We will write $a \lesssim b$ to indicate that $a \leq Cb$, where the constant $C$ is independent of $\eps$, the spatial variables $x,y$, time $t$,  and the exponent $p \in (1, \infty]$ (see \S ~\ref{sec:EE}). The constant $C$ can, however, be dependent on $\gamma$, and therefore also on $\alpha$, and it can depend on $1-\|\bar{w}_0\|_{L^\infty}$ (see \S~\ref{sec:IE}). Furthermore, $C$ will always be independent of our choice of initial data $(w_0,z_0,k_0)$ satisfying the hypotheses specified at the beginning of each section, and it will be independent of our specific choice of $\bar{w}_0$ satisfying the hypotheses listed in \S~\ref{sec:ID}. We will write $a \sim b$ to mean that $a \lesssim b \lesssim a$, and we will further adopt the notation $f = \OO(g)$ to mean that $|f| \lesssim |g|$.


Often below we will have functions $f$ defined on $\TT \times [0,T_*)$ and maps $\Psi: \TT \times [0,T_*) \rightarrow \TT$, and we will use the notation
$$ f\cir \Psi (x,t) : = f(\Psi(x,t), t). $$
We will write $\Psi^{-1}$ to denote the function such that $\Psi^{-1} \cir \Psi (x,t) = \Psi \cir \Psi^{-1} (x,t) = x$ for all $t$ when such an inverse exists. 
For functions $f$ defined on $\TT \times[0,T_*)$ we will write
$$ [f]_{C^{0,\delta}} : = \sup_{\substack{y,y' \in \TT \\ y \neq y'}} \frac{|f(y,t)-f(y',t)|}{|y-y'|^\delta} $$
so that $[\cdot ]_{C^{0,\delta}}$ only corresponds to H\"{o}lder continuity in the spacial variable and $[f ]_{C^{0,\delta}}$ is always implicitly a function of $t$ if $f$ varies in $t$.

\subsection{Statement of the theorem}

\begin{theorem}[Main Result] 
\label{thm:main}
Let $\alpha >0$, $\beta \geq 1$, let $p = \frac{\beta}{\beta-1}\in (1,\infty]$ be the H\"{o}lder conjugate of $\beta$, and let $\bar{w}_0$ be defined as in \S~\ref{sec:ID} for this choice of $\beta$. There exists a sufficiently small $\eps_0 \in (0,1]$, which only depends on $\alpha$ and on $\| \bar w_0\|_{L^\infty}$, such that for any $0< \eps \leq \eps_0$ the following holds. 

Let $w_0, z_0, k_0$ be functions on $\TT$ satisfying
\begin{itemize}
\item $\|w_0 - (1+\bar{w}_0)\|_{C^{1,\frac{1}{\beta}}} < \eps$,
\item $\|z_0\|_{W^{2,p}} < \eps$,
\item $\|k_0-\fint_\TT k_0\|_{W^{2,p}} < \eps$,
\item $w_0 \in W^{2,q}(\TT)$ for some $1< q < \infty$.
\end{itemize}
The maximal time of existence $T_*$ for the solution $(w,z,k)$ of \eqref{eq:euler:RV} is
\begin{equation}
T_* = \tfrac{2}{1+\alpha} + \OO(\eps).
\end{equation}
For such solutions, $w\sim 1$ on $\TT \times [0,T_*]$, and $z,k, \p_y z,$ and $\p_y k$ remain $\OO(\eps)$ on $\TT\times [0,T_*]$ while $\p_y w $ diverges to $-\infty$ as $t \rightarrow T_*$. At the time of blowup, there exists a unique point $y_* \in \TT$ where $w(\cdot, T_*)$ is $C^1$ away from$y_*$ but forms a preshock at $(y_*,T_*)$ with a profile given by
\begin{equation}
\label{eq:thm:expansion:1}
w(y,T_*) = w(y_*, T_*) - \big( (1+\tfrac{1}{\beta})^{\frac{\beta}{\beta+1}} + \OO(\eps)\big) \sgn(y-y_*) |y-y_*|^{\frac{\beta}{\beta+1}} + (1+\OO(\eps)) (y-y_*)
\end{equation}
for $|y-y_*| < \frac{\beta}{\beta+1} + \OO(\eps)$. Additionally, we know that $y_* = 1+\OO(\eps)$.

The variables $z$ and $k$ remain in $C^{1,\frac{1}{\beta+1}}(\TT)$ uniformly up to time $T_*$ with
\begin{equation}
\label{eq:thm:z:k}
\| z(\cdot, t) \|_{C^{1,\frac{1}{\beta+1}}} \lesssim \eps, \qquad \| k(\cdot, t) \|_{C^{1,\frac{1}{\beta+1}}} \lesssim \eps.
\end{equation}

For any fixed $\gamma> 1$, the implicit constants in all of the above statements are independent of our choice of $\eps$ sufficiently small and are independent of our choice of $(w_0,z_0,k_0)$ satisfying our above hypotheses. The implicit constants do depend, however, on our choice of $\gamma >1$.
\end{theorem}

\subsection{Outline of the proof}

We show that the classical solution $(w,z,k)$ of \eqref{eq:euler:RV} with initial data satisfying some of our hypotheses only exists for a finite time $T_*$, and we will  characterize $T_*$ as the time when the flow $\eta$ of the fastest wave speed $\lambda_3$ stops being a diffeomorphism because $\eta_x(\cdot, T_*)$ has a zero. For solutions with appropriate initial data, we will prove that $w,z,k, \p_y z,$ and $\p_y k$ are all bounded on $\TT \times [0,T_*]$ while $\p_y w$ diverges to $-\infty$ as $t$ approaches $T_*$. The key technical step of the paper will be obtaining uniform $L^p$ energy estimates for the functions $w\circ \eta, z\circ \eta, k\circ \eta$, together with their derivatives of order $\leq 2$. Using these $L^p$ energy estimates, we will  show that the unique label $x_* \in \TT$ for which $\eta_x(x_*,T_*) = 0$ is $x_* = 0$. Lastly, we invert the map $x \rightarrow \eta(x,T_*)$ for $x$ near 0 and arrive at the shock profile described in Theorem \ref{thm:main}.

\section{Key identities}
\label{sec:keyid}

In this section, $(w,z,k)$ will be the solution of \eqref{eq:euler:RV} on $\TT \times [0,T_*)$ for given initial data $(w_0,z_0,k_0) \in H^s(\TT)$ for some $s > \frac{3}{2}$.

In general, if $\varphi \in \R$ and $\lambda : = u + \varphi \sigma = \tfrac{1+\varphi}{2} w + \tfrac{1-\varphi}{2} z$, then
\begin{equation}
\p_t \sigma + \lambda \p_y \sigma = - \sigma\big( \tfrac{\alpha-\varphi}{2} \p_y w  + \tfrac{\alpha+\varphi}{2} \p_y z \big).
\end{equation}
Setting $\varphi = -\alpha, 0,$ and $\alpha$ we get
\begin{align}
\label{eq:c}
\sigma \cir\psi & = \sigma _0 e^{-\alpha \int^t_0 \p_y w\circ \psi}, \notag \\
\sigma \cir \phi & = \sigma_0 e^{-\alpha \int^t_0 \p_y \frac{w+z}{2}\circ \phi}, \notag \\
\sigma \cir \eta & = \sigma_0e^{-\alpha \int^t_0 \p_y z\circ \eta}.
\end{align}
Since
\begin{align}
\psi_x & = e^{\int^t_0 \p_y \lambda_1 \circ \psi}, & \phi_x & = e^{\int^t_0 \frac{1}{2}(\p_yw + \p_y z) \circ \phi}, & \eta_x & = e^{\int^t_0 \p_y \lambda_3 \circ \eta} 
\end{align}
It follows that
\begin{align}
\label{eq:psi_x}
\psi_x & = (\tfrac{\sigma_0}{\sigma\circ \psi})^{\frac{1-\alpha}{2\alpha}} e^{\frac{1+\alpha}{2} \int^t_0 \p_y z\circ \psi}, \\
\label{eq:phi_x}
\phi_x & = \big( \tfrac{\sigma_0}{\sigma\circ \phi} \big)^{\frac{1}{\alpha}}, \\
\label{eq:eta_x}
\eta_x & = \big( \tfrac{\sigma_0}{\sigma\circ \eta} \big)^{\frac{1-\alpha}{2\alpha}} e^{\frac{1+\alpha}{2} \int^t_0 \p_y w\circ \eta}.
\end{align}

\begin{align}
\p_t(k \cir \psi) & = - \alpha (\sigma\p_y k )\cir\psi, \\
\label{eq:k_t:2}
\p_t(k\cir \phi) & = 0, \\
\p_t(k \cir \eta) & = \alpha(\sigma \p_y k)\cir \eta.
\end{align}

It follows from \eqref{eq:phi_x} and \eqref{eq:k_t:2} that
\begin{align}
\label{eq:k_y}
\p_y k \cir \phi & = \sigma^{-\frac{1}{\alpha}}_0 k'_0 \sigma^{\frac{1}{\alpha}} \cir \phi.
\end{align}

Define
\begin{align}
\label{def:q^w}
q^w := \p_y w - \tfrac{1}{2\gamma} \sigma\p_y k, 
\qquad 
\mbox{and}
\qquad 
q^z := \p_y z + \tfrac{1}{2\gamma} \sigma\p_y k.
\end{align}

Taking $\p_t$ of $\eta_x q^w\cir \eta$ and $\psi_x q^z\cir\psi$ and then integrating gives us the Duhamel formulas
\begin{align}
\label{eq:q^w}
\eta_x q^w\cir \eta & = e^{\frac{1}{4\gamma} k\circ \eta} \bigg[ (w'_0 - \tfrac{1}{2\gamma} \sigma_0 k'_0)e^{-\frac{1}{4\gamma} k_0} + \tfrac{\alpha}{4\gamma} \int^t_0 \! \! e^{-\frac{1}{4\gamma} k \circ \eta} \eta_x (\sigma \p_y k q^z) \cir \eta \: ds \bigg], \\
\label{eq:q^z}
\psi_x q^z\cir \psi & = e^{\frac{1}{4\gamma} k\circ \psi} \bigg[ (z'_0 + \tfrac{1}{2\gamma} \sigma_0 k'_0)e^{-\frac{1}{4\gamma} k_0} - \tfrac{\alpha}{4\gamma} \int^t_0 \! \! e^{-\frac{1}{4\gamma} k \circ \psi} \psi_x (\sigma \p_y k q^w) \cir \psi \: ds \bigg].
\end{align}


\section{Initial Estimates}
\label{sec:IE}

In this section, $\beta \geq 1$, $\bar{w}_0$ will be a function satisfying the hypotheses described in \S~\ref{sec:ID}, and $(w,z,k)$ will be the solution of \eqref{eq:euler:RV} on $\TT \times [0,T_*)$ for given initial data $(w_0,z_0,k_0)$ satisfying
\begin{itemize}
\item $\|w_0 - (1+\bar{w}_0) \|_{C^1} < \eps$,
\item $\|z_0\|_{C^1} < \eps$,
\item $\|k_0 - \fint_\TT k_0 \|_{C^1} < \eps$,
\item $w_0,z_0, k_0 \in W^{2,q}(\TT)$ for some $1< q \leq \infty$.
\end{itemize}
The last assumption is simply there so that we can apply our local well-posedness theory and Lipschitz continuation criterion described in \S~\ref{sec:LWP}. All implicit constants will be independent of our choice of $(w_0,z_0,k_0)$ satisfying these constraints and independent of our choice of $\beta \geq 1$, but will depend on our choice of $\alpha > 0$.

\begin{lemma}
\label{lem:1:1}
Suppose that $T \in [0, \frac{4}{1+\alpha} \wedge T_*]$ and that for all $(y,t) \in \TT \times [0,T]$ we have
\begin{itemize}
\item $w \sim 1$,
\item $ |k|, |\p_y k|, |z|, |\p_y z| \lesssim \varepsilon$.
\end{itemize}
Then if $\varphi < \alpha$ and $\Psi$ is defined to be the flow of $\lambda : = u + \varphi \sigma$ we have
\begin{itemize}
\item $\eta_x \leq 3+\alpha$,
\item $ \eta_x q^w\cir \eta = \bar{w}'_0 + \mathcal O(\varepsilon)$,
\item $ \int^t_0 |\p_y w\cir \Psi(x,s)| \: ds \lesssim \tfrac{t}{\alpha-\varphi}$,
\end{itemize}
for all $(x,t) \in \TT \times [0,T]$. Note that the implicit constants in our conclusions are independent of our choice of $\varphi < \alpha$ or $T \in [0, \frac{4}{1+\alpha} \wedge T_*]$ and depend only on the initial data and the implicit constants in our hypotheses.
\end{lemma}

\begin{proof}[Proof of Lemma ~\ref{lem:1:1}]
Fix $\varphi <\alpha$ and define
$$ g(x,t) : = \eta^{-1}(\Psi(x,t),t). $$
We compute that
\begin{align}
\label{g_t}
\p_t g(x,t) & = \p_t \eta^{-1}(\Psi(x,t),t) + \p_t \eta^{-1}( \Psi(x,t),t) \Psi_t(x,t) \notag \\
& = \frac{ - \eta_t(g(x,t),t) + \Psi_t(x,t)}{ \eta_x(g(x,t),t)} \notag \\
& = \frac{-\lambda_3 \cir \eta( (g(x,t),t) + \lambda \cir \Psi(x,t)}{ \eta_x(g(x,t),t)} \notag \\
& = \bigg( \frac{-\lambda_3 \cir \eta + \lambda \cir\eta}{\eta_x} \bigg) ( g(x,t), t) \notag \\
& = - (\alpha-\varphi) \frac{\sigma\cir\eta}{\eta_x} ( g(x,t),t).
\end{align}
Note that $\p_t g(x,t) < 0$ everywhere. We also know that
\begin{align*}
 \Psi(x,t) & = x + \int^t_0 \! \! \lambda \cir \Psi(x,\tau) \: d\tau, \hspace{5mm} \text{ and } \\
\Psi(x,t) & = \eta(g(x,t),t) = g(x,t) + \int^t_{0} \! \! \lambda_3 \cir \eta (g(x,t),\tau) \: d\tau, \\ 
\text{so } \hspace{4mm} x-g(x,t) & = \int^t_{0} \! \! \lambda_3 \cir \eta(g(x,t), \tau) - \lambda \cir\Psi(x,\tau) \: d\tau.
\end{align*}

Using our hypotheses, along with \eqref{eq:q^w}, we conclude that that for all $(x,t) \in \mathbb{T} \times [-\varepsilon, T]$, we have
\begin{align*}
\sup_{[-\varepsilon, t]} |\eta_x q^w\cir\eta| & \leq 1+ \OO(\eps) +  \OO(\eps) \sup_{[-\varepsilon, t]} \eta_x.
\end{align*}
Since
\begin{align*}
\eta_x & = 1 + \int^t_0 \! \! \eta_x \p_y \lambda_3\cir \eta \: ds \\
& = 1 + \int^t_0 \! \! \eta_x(\tfrac{1+\alpha}{2}q^w\cir \eta + \OO(\eps) ) \: ds,
\end{align*}
it follows that
\begin{align*}
\sup_{[0,t]} \eta_x & \leq 1 + \tfrac{1+\alpha}{2} t + \OO(\eps t) + \OO(\eps t) \sup_{[0,t]} \eta_x \\
\implies \sup_{[0,t]} \eta_x & \leq \frac{1 + \tfrac{1+\alpha}{2} t + \OO(\eps t)}{1-\OO(\eps t)} \leq 3+\alpha.
\end{align*}
The last inequality is true for $\varepsilon> 0$ taken to be small enough, since $t \leq \frac{4}{1+\alpha}$. Plugging this into  \eqref{eq:q^w} and letting $\varepsilon$ be sufficiently small gives us
\begin{align*}
 q^w\cir \eta & = \frac{\bar{w}'_0 + \OO(\eps)}{\eta_x}.
\end{align*}
It follows that
$$ q^w\cir \Psi(x,t) = -\tfrac{1}{\alpha-\varphi} \p_t g(x,t) \frac{\bar{w}'_0(g(x,t)) + \mathcal O(\eps)}{  \sigma\cir\Psi(x,t)}. $$
Since $\p_t g < 0$, it follows that
$$ |q^w\cir \Psi(x,t)| \lesssim -\tfrac{1}{\alpha-\varphi} \p_t g(x,t). $$
So
$$ \int^t_{0} \! \! |q^w\cir\Psi(x,\tau)| \: d\tau \lesssim \frac{x-g(x,t)}{ \alpha-\varphi} \lesssim \tfrac{t}{\alpha-\varphi} $$
Our result follows immediately from this inequality and our hypotheses.
\end{proof}

\begin{lemma}
\label{lem:1:2}
For all $(y,t) \in \TT \times [0,\frac{4}{1+\alpha} \wedge T_*]$ we have
\begin{itemize}
\item $\tfrac{1}{2} - \OO(\eps) \leq w \leq \tfrac{3}{2} + \OO(\eps)$,
\item $|z|, |k|, |\p_y k| \lesssim \eps$.
\end{itemize}
Further, for all $(x,t) \in \TT\times [0, \frac{4}{1+\alpha} \wedge T_*]$ we have $\phi_x \sim 1$.
\end{lemma}

\begin{proof}
This follows from a simple bootstrap argument. Let $T \in [0, \frac{4}{1+\alpha} \wedge T_*]$ and assume that 
\begin{itemize}
\item $\tfrac{1}{4} \leq w \leq 2$ for all $(y,t) \in \TT \times [0,T]$,
\item $|z|, |k|, |\p_y k| \leq 1$ for all $(y,t) \in \TT \times [0,T]$,
\item $4^{-\tfrac{1}{\alpha}} \leq \phi_x \leq 4^{\tfrac{1}{\alpha}}$ for all $(x,t) \in \TT \times [0,T]$.
\end{itemize}
Then it follows from \eqref{eq:k_t:2} that $k\circ\phi = k_0 = \OO(\eps)$ and $\p_y k \circ \phi = \phi^{-1}_x k'_0 = \OO(\eps)$ for all $(x,t) \in \TT\times [0,T]$. Our bootstrap assumptions tell us that $|\sigma| \lesssim 1$ for $(y,t) \in \TT \times [0,T]$, so $\sigma^2 \p_yk = \OO(\eps)$ for $(y,t) \in \TT \times [0,T]$. It therefore follows from equations \eqref{eq:euler:2} that
\begin{align*}
w\cir \eta & = w_0 + \OO(\eps t) = 1+ \bar{w}_0 + \OO(\eps), \\
z\cir \psi & = z_0 + \OO(\eps t) = z_0 + \OO(\eps), \\
\text{ so } 
\tfrac{1}{2} - \OO(\eps) \leq &  w \leq \tfrac{3}{2} + \OO(\eps), \\
 \text{ and }  z  & = \OO(\eps).
\end{align*}
This implies that $\tfrac{1}{4} - \OO(\eps) \leq  \sigma \leq \tfrac{3}{4} + \OO(\eps)$ for $(y,t) \in \TT \times [0,T]$ and it therefore follows from \eqref{eq:phi_x} that
$$ (3+ \OO(\eps))^{-\tfrac{1}{\alpha}} \leq \phi_x \leq (3+\OO(\eps))^{\tfrac{1}{\alpha}}. $$
\end{proof}

\begin{lemma}
\label{lem:1:3}
For all $(x,t) \in \TT \times [0, \frac{4}{1+\alpha} \wedge T_*]$ we have
\begin{itemize}
\item $\eta_xq^w\circ \eta = \bar{w}_0' + \OO(\eps)$,
\item $\eta_x \leq 3+\alpha$,
\item $\int^t_0 |\p_y w\circ \psi| \lesssim 1$,
\item $ \int^t_0 |\p_y w\circ \phi| \lesssim 1$,
\item $|\p_y z\circ \psi| \lesssim \eps$,
\item $\psi_x \sim 1$.
\end{itemize}
\end{lemma}

\begin{proof}[Proof of Lemma ~\ref{lem:1:3}]
Another simple bootstrap argument. Let $T \in [0, \frac{4}{1+\alpha} \wedge T_*]$ and let us assume that 
$$ |q^z| \leq C\eps $$
where $C > 1$ is a constant to be determined. Since $\p_y k = \OO(\eps)$ and $\sigma \sim 1$ (see Lemma ~\ref{lem:1:2}) , it follows from the bootstrap assumption and \eqref{eq:psi_x} that $\psi_x \sim 1$ provided that $\eps$ is small enough relative to $C$.

The bootstrap assumption and the bounds from Lemma ~\ref{lem:1:2} let allow us to use Lemma ~\ref{lem:1:1} to conclude that $\int^t_0 |q^w\cir\eta| \lesssim 1$. Since $\psi_x \sim 1$ for $(x,t) \in \TT \times [0,T]$, we now conclude from \eqref{eq:q^z} that
$$ |\psi_x q^z\cir\psi - z'_0| \lesssim \eps. $$
It therefore follows that if $C > 1$ is chosen large enough and $\eps$ is chosen small enough the bootstrap argument closes.

The rest of the inequalities are now immediate from Lemma ~\ref{lem:1:1}.
\end{proof}


\begin{proposition}[\bf Lagrangian blowup criterion]
\label{prop:T_*}
The blowup time $T_*$ (defined by the Eulerian blowup criterion~\eqref{eq:LCC}) satisfies
\begin{equation}
\label{eq:T_*}
T_* = \tfrac{2}{1+\alpha} + \OO(\eps),
\end{equation}
and therefore the bounds from the previous lemmas hold up to time $T_*$. Furthermore, the Eulerian blowup criterion~\eqref{eq:LCC} implies the Lagrangian blowup criterion
\begin{equation}
\label{lim:eta_x}
\liminf_{t \rightarrow T_*} \Bigl( \min_{x \in \TT} \eta_x(x,t) \Bigr) = 0,
\end{equation} 
and we have the bound
\begin{equation}
\label{ineq:eta_x}
\inf_{\substack{1\leq |x| \leq \pi \\ 0 \leq t \leq T_*}} \eta_x(x,t) > \tfrac{1}{2}.
\end{equation}
Lastly, we note that the Lagrangian blowup criterion~\eqref{lim:eta_x} also implies the Eulerian blowup criterion~\eqref{eq:LCC}, and so we may use \eqref{lim:eta_x} as the definition of $T_*$.
\end{proposition}

\begin{proof}[Proof of Proposition ~\ref{prop:T_*}]
We know from Lemma ~\ref{lem:1:3} that for $(x,t) \in \TT \times [0,\frac{4}{1+\alpha} \wedge T_*]$ we have $\eta_x q^w\circ \eta = \bar{w}_0' + \OO(\eps)$, so it follows from Lemma \ref{lem:1:2} that
\begin{align}
\eta_x & = 1 + \int^t_0 \! \! \eta_x \p_y \lambda_3 \cir \eta \: ds 
= 1 + \tfrac{1+\alpha}{2} \int^t_0 \! \! \eta_x q^w\cir \eta \: ds + \OO(\eps) 
= 1 + \tfrac{1+\alpha}{2} t\bar{w}_0' + \OO(\eps)
\label{eq:eta_x:approx}
\end{align}
for $(x,t) \in \TT \times [0, \frac{4}{1+\alpha} \wedge T_*]$. It follows that
$$ 0 \leq \eta_x(0,t) = 1-\tfrac{1+\alpha}{2}t + \OO(\eps) $$
for $t \in [0,\frac{4}{1+\alpha},\wedge T_*]$. Therefore, we must have $T_* \leq \tfrac{2}{1+\alpha} + \OO(\eps)$.

Since $T_* = \frac{4}{1+\alpha} \wedge T_*$, it follows that $k,z,\p_yk,$ and $\p_y z$ are $\OO(\eps)$ and $w$ is $\OO(1)$ up to time $T_*$. Therefore, it follows from \eqref{eq:LCC} that
$$ \limsup_{t \rightarrow T_*} \|\p_y w(\cdot, t) \|_{L^\infty} = \infty. $$
Since  $\p_y w\cir \eta = (\bar{w}'_0 + \OO(\eps)) \eta_x^{-1}$, we conclude that 
$$
\infty 
= \limsup_{t\to T_*} \|\p_y w(\cdot,t)\|_{L^\infty} 
= \limsup_{t\to T_*} \|\p_y w \circ \eta (\cdot,t)\|_{L^\infty}
\leq 2 \limsup_{t\to T_*} \|\eta_x^{-1}(\cdot,t)\|_{L^\infty}
\,,
$$
and hence 
\eqref{lim:eta_x} holds.

Since
$$ \eta_x \geq 1- \tfrac{1+\alpha}{2}t + \OO(\eps) $$
for all $(x,t) \in \TT \times [0,T_*)$, we conclude that $T_* \geq \frac{1+\alpha}{2} + \OO(\eps)$, otherwise \eqref{lim:eta_x} would be violated.

To see why \eqref{ineq:eta_x} is true, just note that $\bar{w}'_0(x) \geq 0$ for $1 \leq |x| \leq \pi$, so that $\eta_x \geq 1 + \OO(\eps)$ for $1 \leq |x| \leq \pi$.

Lastly, to see that~\eqref{lim:eta_x} implies~\eqref{eq:LCC}, we show that {\em not}~\eqref{eq:LCC} implies {\em not}~\eqref{lim:eta_x}. More precisely, if there exists $0<M<\infty$ such that 
$$
\int^{T_*}_0 \Bigl(\|\p_y w(\cdot, t)\|_{L^\infty} + \|\p_y z(\cdot, t)\|_{L^\infty} \Bigr)    dt
\leq M,
$$
then \eqref{eq:c} and the triangle inequality gives $e^{-\alpha M} \leq  \frac{\sigma\circ \eta}{\sigma_0}(x,t)  \leq e^{\alpha M}$ for all $(x,t)\in \TT\times [0,T_*]$, which may be combined with \eqref{eq:eta_x} to yield
$e^{-M} \leq \eta_x(x,t) \leq e^{M} $
for all $(x,t)\in \TT\times [0,T_*]$. 
Thus, \eqref{lim:eta_x} fails, concluding the proof.
\end{proof}

Let us now prove an important inequality for $\eta_x q^w\cir\eta$ using \eqref{eq:T_*}. For $\frac{1}{1+\alpha} \leq t \leq T_*$, \eqref{eq:eta_x:approx} tells us that
$$ \bar{w}'_0 = \tfrac{2}{1+\alpha}(\eta_x-1+\OO(\eps) ) \tfrac{1}{t} . $$
Therefore, for $\frac{1}{1+\alpha} \leq t \leq T_*$ we have
\begin{align*}
\eta_x q^w\cir\eta
= \bar{w}'_0 + \OO(\eps) 
& = \tfrac{1}{t} \tfrac{2}{1+\alpha} \eta_x + \tfrac{1}{t} \tfrac{2}{1+\alpha} (-1 + \OO(\eps)) \\
& \leq 2 \eta_x + \tfrac{1}{T_*}\tfrac{2}{1+\alpha} (-1 + \OO(\eps)) \\
& \leq 2\eta_x -1 + \OO(\eps) \\
& \leq -\tfrac{1}{2} + 2\eta_x.
\end{align*}
For $ 0 \leq t \leq \tfrac{1}{1+\alpha}$, \eqref{eq:eta_x:approx} gives 
\begin{align*}
\eta_x 
& = 1+ \tfrac{1+\alpha}{2} t \bar{w}'_0 + \OO(\eps) 
\geq 1- \tfrac{1+\alpha}{2} t + \OO(\eps) 
\geq \tfrac{1}{2} + \OO(\eps).
\end{align*}
Therefore, for $ 0 \leq t \leq \tfrac{1}{1+\alpha}$ we get
\begin{align*}
\eta_x q^w\cir\eta 
& = \bar{w}'_0 + \OO(\eps) 
\leq 1 + \OO(\eps) 
= (1-4\eta_x) + \OO(\eps) + 4\eta_x 
\leq - \tfrac{1}{2} + 4\eta_x.
\end{align*}
We have now proven that
\begin{equation}
\label{ineq:q^w}
\eta_x q^w \cir \eta \leq -\tfrac{1}{2} + 4\eta_x \qquad \forall \: (x,t) \in \TT \times [0, T_*).
\end{equation}

\subsection{An identity for $\p^2_y k$}
\label{sec:k}

Note that since $\sigma$ and $u$ are $C^1$ on $\TT \times [0,T_*)$, it follows from \eqref{eq:phi_x} and the fact that $\phi_t = u\cir \phi$ that $\phi$ is $C^2$ on $\TT \times [0,T_*)$. Since $\phi_x \sim 1$, it follows that $\phi^{-1}$ is $C^2$ on $\TT \times[0,T_*)$. It follows $k = k_0 \cir \phi^{-1} \in W^{2,q}_{\loc} (\TT \times [0,T_*))$.

Differentiating \eqref{eq:k_y} in $x$  and using \eqref{eq:phi_x} gives us
\begin{align*}
(\p^2_y k - \tfrac{1}{\alpha} \sigma^{-1} \p_y \sigma \p_y k)\cir \phi & = \phi^2_x (k''_0 -\tfrac{1}{\alpha}\sigma^{-1}_0 \sigma'_0 k'_0).
\end{align*}
Therefore,
\begin{align}
\label{eq:k_yy}
\p^2_y k & = \tfrac{1}{\alpha} \sigma^{-1} \p_y \sigma \p_y k + \big[ \phi^2_x (k''_0 -\tfrac{1}{\alpha}\sigma^{-1}_0 \sigma'_0 k'_0)\big] \cir \phi^{-1}.
\end{align}


\section{Identities along the 3-characteristic}
\label{sec:3charid}

In this section, $\beta \geq 1$, $\bar{w}_0$ will be a function satisfying the hypotheses described in \S~\ref{sec:ID}, and $(w,z,k)$ will be the solution of \eqref{eq:euler:RV} on $\TT \times [0,T_*)$ for given initial data $(w_0,z_0,k_0)$ satisfying
\begin{itemize}
\item $\|w_0 - (1+\bar{w}_0) \|_{C^1} < \eps$,
\item $\|z_0\|_{C^1} < \eps$,
\item $\|k_0 - \fint_\TT k_0 \|_{C^1} < \eps$,
\item $w_0,z_0, k_0 \in W^{2,q}(\TT)$ for some $1< q \leq \infty$.
\end{itemize}
All implicit constants will be independent of our choice of $(w_0,z_0,k_0)$ satisfying these constraints and independent of our choice of $\beta \geq 1$, but will depend on our choice of $\alpha > 0$. 

Adopt the notation
\begin{align*}
W & : = w\cir \eta, & \W & : = q^w\cir \eta, \\
Z & : = z\cir \eta, & \Z & : = q^z \cir \eta, \\
K & : = k \cir \eta, & \K & : = \p_y k \cir \eta, \\
\Sigma & : = \sigma\cir \eta, & \C & : = \p_y \sigma\cir \eta.
\end{align*}

Using \eqref{eq:euler:RV} and the equation $\eta_t = \lambda_3 \circ \eta$ gives us
\begin{align}
\label{eq:K_t}
K_t & = \alpha \Sigma \K, \\
\label{eq:Z_t}
Z_t & = \Sigma(2\alpha \Z- \tfrac{1}{2\gamma} K_t) \\
\label{eq:C_t}
\Sigma_t & = - \Sigma(\alpha \Z - \tfrac{1}{2\gamma} K_t) \\
\label{eq:eta_xt}
\eta_{xt} & = \eta_x (\tfrac{1+\alpha}{2} \W + \tfrac{1-\alpha}{2} \Z + \tfrac{1}{2\gamma} K_t), 
\end{align}

Taking $\p_x$ of \eqref{eq:K_t} and using gives us
\begin{align}
\label{eq:K_xt}
K_{xt} & =  \eta_x(\alpha \C \K + \alpha \Sigma \p^2_y k \cir \eta) 
\end{align}
Since $K_{xt} = \eta_{xt} \K + \eta_x \K_t$, it follows from \eqref{eq:k_yy},  \eqref{eq:eta_xt}, and \eqref{eq:K_xt} that 
\begin{align}
\label{eq:K_yt}
\K_t 
= \tfrac{K_{xt}-\eta_{xt}\K}{\eta_x} 
&  = -\K\Z + \tfrac{1}{2\gamma} \Sigma \K^2 + \alpha \Sigma \big[ \phi^2_x(k''_0- \tfrac{1}{\alpha} \sigma^{-1}_0 \sigma'_0 k'_0 ) \big] \cir \phi^{-1} \cir \eta \notag \\
& = \OO(\eps) + \OO(k''_0 \circ \phi^{-1} \cir \eta).
\end{align}
Therefore, taking $\p_t$ of \eqref{eq:K_t} gives us
\begin{align}
\label{eq:K_tt}
K_{tt} & = \OO(\eps + k''_0 \circ \phi^{-1} \cir \eta).
\end{align}
Taking $\p_t$ of \eqref{eq:C_t} and using \eqref{eq:K_tt} yields
\begin{align}
\label{eq:C_tt}
\Sigma_{tt} & =- \alpha \Sigma \Z_t + \OO(\eps + k''_0 \circ \phi^{-1} \cir \eta).
\end{align}

Taking $\p_t$ of \eqref{eq:q^w} gives us
\begin{align}
\label{eq:p_teta_xW}
\p_t( \eta_x \W) & = \tfrac{1}{4\gamma} \eta_x K_t(\W + \Z).
\end{align}
Taking $\p_t$ of \eqref{eq:eta_xt} and using \eqref{eq:K_tt} and \eqref{eq:p_teta_xW} produces
\begin{align}
\label{eq:eta_xtt}
\eta_{xtt} = \eta_x\big[ \tfrac{1+\alpha}{2} \W (\tfrac{1-\alpha}{2} \Z + \tfrac{3}{4\gamma} K_t) + \tfrac{1-\alpha}{2} \Z_t + \OO(\eps+ k''_0 \circ \phi^{-1} \cir \eta) \big].
\end{align}

It is immediate that
\begin{align*}
\eta_x \Z & = Z_x + \tfrac{1}{2\gamma} \Sigma K_x. 
\end{align*}
Taking $\p_t$ of this equation and using \eqref{eq:Z_t} gives us
\begin{align}
\label{eq:important:1}
 \p_t(\eta_x \Z) & = 2\alpha \p_x(\Sigma \Z) + \tfrac{1}{2\gamma} (\Sigma_t K_x - \Sigma_x K_t).
\end{align}
Since
\begin{align*}
\p_t(\eta_x \Z) & = \eta_{xt} \Z + \eta_x \Z_t,
\end{align*}
it follows from \eqref{eq:eta_xt} and \eqref{eq:important:1} that
\begin{align}
\label{eq:CZ_x}
2\alpha \Sigma \Z_x & = \eta_x \Z_t - 2\alpha \Sigma_x \Z + \tfrac{1}{2\gamma} (\Sigma_x K_t - \Sigma_t K_x)  + \eta_{xt} \Z \notag \\
& = \eta_x \Z_t - 2\alpha \Sigma_x \Z + \tfrac{1}{2\gamma} (\Sigma_x K_t - \Sigma_t K_x)  
+ \eta_x\Z( \tfrac{1+\alpha}{2} \W + \tfrac{1-\alpha}{2} \Z + \tfrac{1}{2\gamma} K_t).
\end{align}
Taking $\p_x$ of \eqref{eq:C_t} gives us
\begin{align}
\label{eq:C_tx}
\Sigma_{tx} & = -\Sigma_x(\alpha \Z - \tfrac{1}{2\gamma} K_t) - \Sigma(\alpha \Z_x - \tfrac{1}{2\gamma} K_{tx}).
\end{align}
So it follows that taking $\p_t$ of \eqref{eq:important:1} \footnote{ For those concerned that we may not have enough derivatives in $L^2$ to do this, we can circumnavigate this issue by first taking data in $U_\eps \cap C^\infty$, deriving the estimates from \S~\ref{sec:EE}-\ref{sec:approx} in this case, and then passing the estimates to the limit to all solutions with data in $V_\eps$.} and using \eqref{eq:C_tx}, \eqref{eq:C_t}, \eqref{eq:eta_xt}, \eqref{eq:eta_xtt}, \eqref{eq:C_tt}, \eqref{eq:K_xt}, \eqref{eq:K_tt}, and \eqref{eq:CZ_x} results in
\begin{align}
\label{eq:important:2}
\eta_x \Z_{tt}& = 2 \alpha \Sigma \Z_{xt} + \eta_x \Z_t(-\W + \OO(\eps)) + \eta_x\big( \OO(\eps) \W + \OO(\eps) + \OO(\eps k''_0 \cir \phi^{-1} \cir \eta) \big).
\end{align}
This identity is the main computational step for the energy estimates in the next section.

Similar computations to those used to obtain \eqref{eq:K_t}--\eqref{eq:eta_xt} allow us to write
\begin{align*}
\Z_t & = \big( -\p_y\lambda_3q^z + 2\alpha \sigma \p^2_y z  
+ \alpha (2\p_y \sigma-\tfrac{1}{2\gamma} \sigma \p_yk) \p_y z 
+ \tfrac{3\alpha}{2\gamma} \sigma \p_y \sigma \p_y k + \tfrac{\alpha}{\gamma} \sigma^2 \p^2_y k \big) \cir \eta.
\end{align*}
This gives us
\begin{align}
\label{eq:Z_t:0}
\Z_t(\cdot, 0) & = - (\tfrac{1+\alpha}{2}w'_0 + \tfrac{1-\alpha}{2} z'_0) (z'_0 + \tfrac{1}{2\gamma} \sigma_0 k'_0) \notag \\
&\qquad  + 2\alpha \sigma_0 z''_0 + \alpha( 2\sigma'_0 - \tfrac{1}{2\gamma} \sigma_0 k'_0 ) z'_0 
+ \tfrac{3\alpha}{2\gamma} \sigma_0 \sigma'_0 k'_0 + \tfrac{\alpha}{\gamma} \sigma^2_0 k''_0.
\end{align}
So 
\begin{align}
\label{ineq:Z_t:0}
|\Z_t(x,0)| \lesssim  |z''_0| + |k''_0| + \eps.
\end{align}
This inequality will be utilized in the energy estimates below.


\section{Energy estimates}
\label{sec:EE}

In this section, $1< p < \infty$, $\beta \geq 1$, $\bar{w}_0$ will be a function satisfying the hypotheses described in \S~\ref{sec:ID}, and $(w,z,k)$ will be the solution of \eqref{eq:euler:RV} on $\TT \times [0,T_*)$ for given initial data $(w_0,z_0,k_0)$ satisfying
\begin{itemize}
\item $\|w_0 - (1+\bar{w}_0) \|_{C^1} < \eps$,
\item $\|z_0\|_{C^1} < \eps$,
\item $\|k_0 - \fint_\TT k_0 \|_{C^1} < \eps$,
\item $w_0 \in W^{2,q}(\TT)$ for some $1< q< \infty$,
\item $ z_0, k_0 \in W^{2,p}(\TT)$.
\end{itemize}
All implicit constants will be independent of our choice of $p \in (1,\infty)$, independent of our choice of $(w_0,z_0,k_0)$ satisfying these constraints, and independent of our choice of $\beta \geq 1$, but will depend on our choice of $\alpha > 0$.

Let $b > 0$ be a parameter to be determined, and suppose that the quantity
\begin{equation}
E_p(t) : = \int_\TT \! \! \Sigma^{-bp}(x,t) \eta_x(x,t) |\Z_t(x,t)|^p \: dx
\end{equation}
defines a finite, differentiable function for $t \in [0,T_*)$. \footnote{ We will address this assumption below. See Remark \ref{remark}}

Using the identities from \S ~\ref{sec:3charid}, we have
\begin{align*}
\dot{E}_p & = -bp\int \! \! \tfrac{\Sigma_t}{\Sigma} \Sigma^{-bp} \eta_x |\Z_t|^p + \int \! \! \Sigma^{-bp} \eta_{xt} |\Z_t|^p + p \int\! \! \Sigma^{-bp} \eta_x \sgn(\Z_t) |\Z_t|^{p-1} \Z_{tt} \\
& = \int \! \! \Sigma^{-bp} \eta_x |Z_t|^p \big[ \tfrac{1+\alpha}{2} \W + \OO( (1+b)(1+p) \eps) \big] + p \int\! \! \Sigma^{-bp} \eta_x \sgn(\Z_t) |\Z_t|^{p-1} \Z_{tt} \\
& = \int \! \! \Sigma^{-bp} \eta_x |Z_t|^p \big[ \tfrac{1+\alpha}{2} \W + \OO( (1+b)(1+p) \eps) \big] + 2\alpha p \int \! \! \Sigma^{1-bp} \sgn(\Z_t) |\Z_t|^{p-1} \Z_{xt} \\& + p \int \! \! \Sigma^{-bp} \eta_x |Z_t|^p [-\W + \OO(\eps)] + p \int \! \! \Sigma^{-bp} \eta_x \sgn(\Z_t) |Z_t|^{p-1} \OO(\eps) \big[ \W + \OO(1) + \OO(k''_0\cir\phi^{-1}\cir\eta)\big] \\
& = \int \Sigma^{-bp} \eta_x |\Z_t|^p \big[ (\tfrac{1+\alpha}{2} -p) \W + \OO( (1+p)(1+b)\eps)) \big] \\
& +2\alpha p \int \! \! \Sigma^{1-bp} \sgn(\Z_t) |\Z_t|^{p-1} \Z_{xt}  \\
& + p \int \! \! \Sigma^{-bp} \eta_x \sgn(\Z_t) |Z_t|^{p-1} \OO(\eps) \big[ \W + \OO(1) + \OO(k''_0\cir\phi^{-1}\cir\eta)\big]. 
\end{align*}
Since
\begin{align*}
\Sigma^{1-bp} \sgn(\Z_t) |\Z_t|^{p-1} \Z_{xt} & = \tfrac{1}{p} \Sigma^{1-bp} \p_x \big( |\Z_t|^p \big) \\
& = \tfrac{1}{p} \p_x\big(\Sigma^{1-bp} |\Z_t|^p \big) - \tfrac{1-bp}{p} \Sigma^{-bp} \Sigma_x |\Z_t|^p \\ 
& = \tfrac{1}{p} \p_x\big(\Sigma^{1-bp} |\Z_t|^p \big) - \tfrac{1-bp}{2p} \Sigma^{-bp} \eta_x |\Z_t|^p( \W + \OO(\eps)),
\end{align*}
it follows that
\begin{align*}
\tfrac{1}{p}\dot{E}_p & = \int \! \! \Sigma^{-bp} \eta_x |\Z_t|^p \big[ (\alpha b-1 + \tfrac{1-\alpha}{2p}) \W + \OO( (1+b) \eps) \big]  \\
& + \int \! \! \Sigma^{-bp} \eta_x \sgn(\Z_t) |\Z_t|^{p-1} \OO(\eps) \big[ \W + \OO(1) + \OO(k''_0\cir\phi^{-1}\cir\eta) \big] \\
& = \int \! \!  \Sigma^{-bp} \eta_x \W\big[ (\alpha b-1 + \tfrac{1-\alpha}{2p}) |\Z_t|^p + \OO(\eps) |\Z_t|^{p-1} \big] \\
& + \OO((1+b)\eps)E_p + \int \! \! \Sigma^{-bp} \eta_x |\Z_t|^{p-1} \OO(\eps k''_0\cir\phi^{-1}\cir \eta) \\
& \leq \int \! \!  \Sigma^{-bp} \eta_x \W\big[ (\alpha b-1 + \tfrac{1-\alpha}{2p}) |\Z_t|^p + \OO(\eps) |\Z_t|^{p-1} \big] \\
& + \OO((1+b)\eps)E_p + \tfrac{p-1}{p} \OO(\eps) E_p + \tfrac{\OO(\eps)}{p} \int \Sigma^{-bp} \eta_x |\OO(k''_0\cir\phi^{-1}\cir \eta)|^p\\
& = \int \! \!  \Sigma^{-bp} \eta_x \W\big[ (\alpha b-1 + \tfrac{1-\alpha}{2p}) |\Z_t|^p + \OO(\eps) |\Z_t|^{p-1} \big] \\
& + \OO((1+b)\eps)E_p + \OO\big( \frac{\eps}{p} \|\Sigma^{-b}\|^p_{L^\infty_x} \|k''_0\|^p_{L^p_x} \bigg).
\end{align*}

It's easy to check that
\begin{align*}
(\alpha b-2+ \tfrac{3-\alpha}{2p}) |\Z_t|^p + \tfrac{1}{p} \eps^p|\OO(1)|^p 
&\leq (\alpha b-1 + \tfrac{1-\alpha}{2p}) |\Z_t|^p + \OO(\eps) |\Z_t|^{p-1} 
\notag\\
&\leq (\alpha b - \tfrac{1+\alpha}{2p}) |\Z_t|^p + \tfrac{1}{p} \eps^p|\OO(1)|^p,
\end{align*}
so
\begin{align*}
&\eta_x \W\big[ (\alpha b-1 + \tfrac{1-\alpha}{2p}) |\Z_t|^p + \OO(\eps) |\Z_t|^{p-1} \big] 
\notag\\
& = \eta_x (\W)_+\big[ (\alpha b-1 + \tfrac{1-\alpha}{2p}) |\Z_t|^p + \OO(\eps) |\Z_t|^{p-1} \big] \\
&\qquad  - \eta_x (\W)_-\big[ (\alpha b-1 + \tfrac{1-\alpha}{2p}) |\Z_t|^p + \OO(\eps) |\Z_t|^{p-1} \big] \\
& \leq \eta_x(\W)_+ \big[ (\alpha b - \tfrac{1+\alpha}{2p}) |\Z_t|^p + \tfrac{1}{p} \eps^p|\OO(1)|^p\big] \\
&\qquad  - \eta_x(\W))_-\big[ (\alpha b - 2 + \tfrac{3-\alpha}{2p}) |\Z_t|^p - \tfrac{1}{p}\eps^p |\OO(1)|^p \big] \\
& = \big[ (\alpha b - \tfrac{1+\alpha}{2p}) \mathbf{1}_{\{\W> 0\}}+ (\alpha-2+\tfrac{3-\alpha}{2p}) \mathbf{1}_{\{\W< 0\}} \big] |\Z_t|^p \eta_x \W \\
&\qquad  + \big[ (\alpha b - \tfrac{1+\alpha}{2p}) \mathbf{1}_{\{\W> 0\}} + (\alpha b-2 + \tfrac{3-\alpha}{2p} \mathbf{1}_{\{\W< 0\}} \big] \tfrac{1}{p} \eps^p |\OO(1)|^p \eta_x |\W|.
\end{align*}
Now choose $b$ to be $b = \tfrac{2}{\alpha} + \frac{1}{2}$. The above inequality becomes
\begin{align*}
 \eta_x \W\big[ (\alpha b-1 + \tfrac{1-\alpha}{2p}) |\Z_t|^p + \OO(\eps) |\Z_t|^{p-1} \big] & \leq \big[ \tfrac{4p-1+(p-1)\alpha}{2p} \mathbf{1}_{\{\W > 0\}} + \tfrac{3+ (p-1)\alpha}{2p} \mathbf{1}_{\{ \W < 0\}} \big] |\Z_t|^p \eta_x \W \\
& + \tfrac{4p-1+(p-1)\alpha}{2p^2} \eps^p |\OO(1)|^p \eta_x |\W|.
\end{align*}
Using \eqref{ineq:q^w} and the fact that $\eta_x \W = \OO(1)$, we get
\begin{align*}
\eta_x \W\big[ (\alpha b-1 + \tfrac{1-\alpha}{2p}) |\Z_t|^p + \OO(\eps) |\Z_t|^{p-1} \big] & \leq - \tfrac{3+(p-1)\alpha}{4p} |\Z_t|^p  \\
& + \tfrac{4(4p-1+(p-1)\alpha)}{2p} \eta_x |\Z_t|^p + \tfrac{4p-1+(p-1)\alpha}{2p^2}\eps^p |\OO(1)|^p |\OO(1)|.
\end{align*}
Since $b = \tfrac{2}{\alpha} + \frac{1}{2} = \OO(1)$, we have $\Sigma^{-b} = \OO(1)$. Using this last inequality, we now have
\begin{align}
\label{ineq:EE:1}
&\dot{E}_p + \tfrac{3+(p-1)\alpha}{4} \int \Sigma^{-bp} |\Z_t|^p 
\notag\\
& \leq \big[ \tfrac{4(4p-1+(p-1)\alpha)}{2} + \OO(\eps)\big] E_p + \big[\tfrac{4p-1+(p-1)\alpha}{2p} + \eps + \eps \|k''_0\|^p_{L^p_x}\big]\eps^p |\OO(1)|^p |\OO(1)| 
\notag \\
& \leq  \tfrac{17+\alpha}{2}p E_p + \big[ \tfrac{4+\alpha}{2} + \eps + \eps \|k''_0\|^p_{L^p_x} \big] \eps^p |\OO(1)|^p |\OO(1)|.
\end{align}

ODE comparison and \eqref{eq:T_*} now give us
\begin{align*}
E_p(t) & \leq E_p(0) e^{\tfrac{17+\alpha}{2} p t} + \big[ \tfrac{4+\alpha}{2} + \eps + \eps \|k''_0\|^p_{L^p_x} \big] \eps^p |\OO(1)|^p |\OO(1)| \frac{e^{\tfrac{17+\alpha}{2} p t}-1}{\tfrac{17+\alpha}{2}p} \\
& \leq \bigg( E_p(0) + \frac{ 4+ \alpha + 2 \eps + 2\eps \|k''_0\|^p_{L^p_x}}{(17+\alpha)p} \eps^p |\OO(1)|^p |\OO(1)| \bigg) e^{\tfrac{3(17+\alpha)}{1+\alpha} p } .
\end{align*}
So
\begin{align*}
E_p(t)^{1/p} & \leq \bigg( E_p(0)^{1/p} + \eps (1+ \|k''_0\|_{L^p_x})|\OO(1)| \bigg) |\OO(1)|.
\end{align*}
Using \eqref{ineq:Z_t:0} to bound $E_p(0)^{1/p}$ now gives us
\begin{align}
\label{ineq:EE:2}
E_p(t)^{1/p} & \lesssim \|z''_0\|_{L^p_x} + \|k''_0\|_{L^p_x} + \eps.
\end{align}

Plugging \eqref{ineq:EE:2} into \eqref{ineq:EE:1} now gives us
\begin{align}
\label{ineq:EE:Z_t}
\| \Z_t\|_{L^p_x}  & \lesssim \|z''_0\|_{L^p_x} + \|k''_0\|_{L^p_x} + \eps.
\end{align}
Equation \eqref{eq:CZ_x} tells us that
\begin{equation}
2\alpha \Sigma \Z_x  = \eta_x \Z_t + \OO(\eps),
\end{equation}
so we conclude that
\begin{align}
\label{ineq:EE:Z_x}
\| \Z_x\|_{L^p_x}  & \lesssim \|z''_0\|_{L^p_x} + \|k''_0\|_{L^p_x} + \eps.
\end{align}

\subsection{The case $p= \infty$}

In the case where $z_0, k_0 \in W^{2,\infty}(\TT)$, we have just proven that
\begin{align*}
\| \Z_x\|_{L^p_x}+ \| \Z_t\|_{L^p_x}  & \lesssim \|z''_0\|_{L^p_x} + \|k''_0\|_{L^p_x} + \eps
\end{align*}
for all $1< p < \infty$ with constants independent of $p$. Therefore, sending $p \rightarrow \infty$ gives us
\begin{align}
\label{ineq:EE:infty}
\| \Z_x\|_{L^\infty_x}+ \| \Z_t\|_{L^\infty_x}  & \lesssim \|z''_0\|_{L^\infty_x} + \|k''_0\|_{L^\infty_x} + \eps.
\end{align}


\section{Estimates along the 3-characteristic}
\label{sec:approx}

Let $\beta \geq 1$, $1 < p \leq \infty$, $\bar{w}_0$ will be a function satisfying the hypotheses described in \S~\ref{sec:ID}, and let $p'$ be the H\"{o}lder conjugate of $p$, i.e. $\frac{1}{p'} + \frac{1}{p} =1$, with $p'= 1$ when $p=\infty$. Suppose that
\begin{itemize}
\item $ \|w_0 - (1+\bar{w}_0)\|_{C^1} < \eps$,
\item $[w'_0]_{C^{0,\frac{1}{p'}}} < 2$,
\item $w_0 \in W^{2,q}(\TT)$ for some $1< q \leq \infty$,
\item $\|z_0\|_{W^{2,p}} < \eps$,
\item $\|k_0-\fint_\TT k_0\|_{W^{2,p}} < \eps$.
\end{itemize}
Our energy estimates \eqref{ineq:EE:Z_x}, \eqref{ineq:EE:infty} from the previous section let us conclude that
\begin{align}
\label{ineq:q^z:holder}
[\Z]_{C^{0,\frac{1}{p'}}} \leq \|\Z_x\|_{L^p_x} \lesssim \eps.
\end{align}

Recall that 
\begin{align*}
\eta_x \W & = w'_0+ (e^{\frac{1}{4\gamma} (K-k_0)}-1) w'_0 - \tfrac{e^{\frac{1}{4\gamma}(K-k_0)}}{2\gamma} \sigma_0 k'_0 
+ \tfrac{\alpha}{4\gamma} e^{\frac{1}{4\gamma} K} \int^t_0 e^{-\tfrac{1}{4\gamma}K} \Sigma K_x \Z \: ds, \\
K_x & = \Sigma^{\frac{1}{\alpha}} \eta_x( \tfrac{k'_0}{\sigma^{\frac{1}{\alpha}}} \cir \phi^{-1} \cir \eta), \\
\eta_x & = 1 + \tfrac{1+\alpha}{2} \int^t_0 \eta_x \W \: ds + \tfrac{1-\alpha}{2} \int^t_0 \eta_x \Z \: ds + \tfrac{\alpha}{2\gamma} \int^t_0 \Sigma K_x \: ds.
\end{align*}
Using these equations and \eqref{ineq:q^z:holder}, it is immediate that
\begin{align*}
[\eta_x\W-w'_0]_{C^{0,\frac{1}{p'}}} & \lesssim \eps [w'_0]_{C^{0,\frac{1}{p'}}} + \eps +  \eps(\sup_{[0,t]}[K_x]_{C^{0,\frac{1}{p'}}}), \\
[K_x]_{C^{0,\frac{1}{p'}}} & \lesssim \eps + \eps [\eta_x]_{C^{0,\frac{1}{p'}}} +[ \tfrac{k'_0}{\sigma^{\frac{1}{\alpha}}_0} \cir \phi^{-1} \cir \eta]_{C^{0,\frac{1}{p'}}}, \\
[\eta_x]_{C^{0,\frac{1}{p'}}} & \leq t\sup_{[0,t]} \big( \tfrac{1+\alpha}{2}[\eta_x\W-w'_0]_{C^{0,\frac{1}{p'}}}+ \tfrac{1+\alpha}{2}[w'_0]_{C^{0,\frac{1}{p'}}} + \OO(\eps) [\eta_x]_{C^{0,\frac{1}{p'}}}+ \OO(\eps) + [K_x]_{C^{0,\frac{1}{p'}}} \big).
\end{align*}
We know that 
\begin{equation}
\label{ineq:k'_0:holder}
 [ \tfrac{k'_0}{\sigma^{\frac{1}{\alpha}}_0} \cir \phi^{-1} \cir \eta]_{C^{0,\frac{1}{p'}}} \leq \big\|\tfrac{\eta_x}{\phi_x}(\sigma^{-1/\alpha}_0 k''_0 - \tfrac{1}{\alpha} \sigma^{-1-1/\alpha}_0 \sigma'_0 k'_0) \cir \phi^{-1}\cir \eta \big\|_{L^p_x} \lesssim \eps,
\end{equation}
so, using the fact that $T_* = \frac{2}{1+\alpha} + \OO(\eps)$, a simple bootstrap argument lets us conclude that
\begin{align}
[\eta_x\W-w'_0]_{C^{0,\frac{1}{p'}}} & = \OO(\eps(1 + [w'_0]_{C^{0,\frac{1}{p'}}}) ), \\
[K_x]_{C^{0,\frac{1}{p'}}} & = \OO(\eps), \\
[\eta_x]_{C^{0,\frac{1}{p'}}} & \leq (1+\OO(\eps)) [w'_0]_{C^{0,\frac{1}{p'}}} + \OO(\eps).
\end{align}

Since $W_x = \eta_x\W + \frac{1}{2\gamma}\Sigma K_x$ and $Z_x = \eta_x \Z - \frac{1}{2\gamma} \Sigma K_x$, we conclude that
\begin{equation*}
[W_x-w'_0]_{C^{0,\frac{1}{p'}}} = \OO(\eps) \qquad \text{ and } \qquad  [Z_x]_{C^{0,\frac{1}{p'}}} = \OO(\eps).
\end{equation*}

We know that 
\begin{equation*}
\K =  \Sigma^{\frac{1}{\alpha}} ( \tfrac{k'_0}{\sigma^{\frac{1}{\alpha}}} \cir \phi^{-1} \cir \eta),
\end{equation*}
so it follows from \eqref{ineq:k'_0:holder} that 
\begin{align}
\label{ineq:k_y:holder}
[\K]_{C^{0,\frac{1}{p'}}} & = \OO(\eps).
\end{align}
Using this along with \eqref{eq:K_t} gives us
\begin{equation}
\label{ineq:K_t:holder}
[K_t]_{C^{0,\frac{1}{p'}}} = \OO(\eps).
\end{equation}
Using \eqref{ineq:q^z:holder} and \eqref{ineq:K_t:holder} in \eqref{eq:Z_t} and \eqref{eq:C_t} gives us
\begin{align}
\label{ineq:Z_t:holder}
[Z_t]_{C^{0,\frac{1}{p'}}} & = \OO(\eps), \\
\label{ineq:C_t:holder}
[\Sigma_t]_{C^{0,\frac{1}{p'}}} & = \OO(\eps).
\end{align}
Since $W = 2\Sigma + Z$, it follows that
\begin{equation}
\label{ineq:W_t:holder}
[W_t]_{C^{0,\frac{1}{p'}}}  = \OO(\eps).
\end{equation}

Putting it all together, we have 
\begin{align}
\label{estimates:holder}
\|W-w_0\|_{C^{1,\frac{1}{p'}}} \,,
\|Z\|_{C^{1,\frac{1}{p'}}} \,,  
\|K\|_{C^{1,\frac{1}{p'}}} \,,
\|W_t\|_{C^{0,\frac{1}{p'}}} \,, 
\|Z_t\|_{C^{0,\frac{1}{p'}}} \,, 
\|K_t\|_{C^{0,\frac{1}{p'}}} &= \OO(\eps) \,,
\notag\\ 
\|\Sigma_t\|_{C^{0,\frac{1}{p'}}} \,,
\|\Z\|_{C^{0,\frac{1}{p'}}} \,, 
\|\K\|_{C^{0,\frac{1}{p'}}} &= \OO(\eps).
\end{align}

\subsection{Remarks} 
\label{remark}

Note that while the hypotheses of Theorem \ref{thm:main} preclude $w_0$ from being $C^2$, the hypotheses stated at the beginnings of sections \S~\ref{sec:IE}-\ref{sec:approx} allow for $w_0,z_0,$ and $k_0$ to all be $C^\infty$. One can check that if we start with initial data $(w_0,z_0,k_0)$ satisfying the hypotheses of \S~\ref{sec:IE}-\ref{sec:approx} , we can mollify $w_0,z_0,$ and $k_0$ to produce a sequence of smooth functions satisfying the hypotheses of \S~\ref{sec:IE}-\ref{sec:approx}. In particular, the estimates \eqref{estimates:holder} will hold uniformly for this sequence, and will provide us with compactness which we can use to pass the H\"{o}lder estimates to the limit. The fact that we can do this allows us to circumvent any concerns that the function $E_p(t)$ used in \S~\ref{sec:EE} wasn't a-priori known to be differentiable or even finite.

We should also note that the hypothesis that $\beta \geq 1$ wasn't utilized anywhere in \S~\ref{sec:IE}-\ref{sec:approx}, and the results of these sections still hold for other choices of $\beta > 0$, albeit with the possibility of $\beta$-dependence being introduced into some of the implicit constants and the possibility of the range of sufficiently small $\eps > 0$ being $\beta$-dependent. All implicit constants should be uniform in $\beta \geq r$ for any $r > 0$, but without a lower bound on $\beta$ we may not have uniform implicit constants.


\section{Proving the theorem}
\label{sec:pf}

Let $\beta \geq 1$, and let $p \in (1, \infty]$ be the H\"{o}lder conjugate of $\beta$. Suppose
\begin{itemize}
\item $\|w_0-(1+\bar{w}_0)\|_{C^{1,\frac{1}{\beta}}} < \eps$,
\item $\|z_0\|_{W^{2,p}} < \eps$,
\item $\|k_0-\fint_\TT k_0\|_{W^{2,p}} < \eps$,
\item $w_0 \in W^{2,q}$ for some $1< q< \infty$.
\end{itemize}
Then we know that
\begin{align*}
[W_x-\bar{w}'_0]_{C^{0,\frac{1}{\beta}}} = \OO(\eps),
\end{align*}
so for $|x| \leq 1$ we have
\begin{align}
W_x & = \bar{w}'_0 + (W_x-\bar{w}_0') \notag \\
& = -1 + |x|^{\frac{1}{\beta}} + (W_x(0,t)+1) + \OO(\eps)|x|^{\frac{1}{\beta}} \notag \\
\label{eq:W_x:thm:1}
& = W_x(0,t) + (1+\OO(\eps)) |x|^{\frac{1}{\beta}}.
\end{align}
Therefore, for $|x| \leq 1$ our H\"{o}lder estimates from the previous section give us
\begin{align}
\eta_x(x,t) & = 1 + \int^t_0 \p_x(\lambda_3 \cir \eta) (x,s) \: ds \notag \\
& = 1+ \tfrac{1+\alpha}{2} \int^t_0 W_x(x,s) \: ds + \frac{1-\alpha}{2} \int^t_0 Z_x(x,s) \: ds \notag \\
& = \eta_x(0,t) + \tfrac{1+\alpha}{2} \int^t_0 (W_x(x,s)-W_x(0,s)) \: ds + \tfrac{1-\alpha}{2} \int^t_0 (Z_x(x,s)-Z_x(0,s)) \: ds \notag \\
\label{eq:eta_x:min}
& =\eta_x(0,t) + (\tfrac{1+\alpha}{2} + \OO(\eps)) |x|^{\frac{1}{\beta}} t.
\end{align}
So $x = 0$ is the unique minimizer of $\eta_x(\cdot ,t)$ over $|x| \leq 1$ for all $t > 0$. It now follows from \eqref{lim:eta_x} and \eqref{ineq:eta_x} that $\eta_x(\cdot, T_*)$ has a unique zero at $x = 0$.

For the remainder of this section, let us restrict our attention to labels $x$ with $|x| \leq 1$. Let $y = \eta(x,T_*)$. We conclude from \eqref{eq:eta_x:min} and \eqref{eq:T_*} that
\begin{align*}
y 
= \eta(0,T_*) + \int^x_0 \eta_x(x',T_*) \: dx' 
& =: y_* + \int^x_0 \eta_x(x',T_*)-\eta_x(0,T_*) \: dx' \\
& = y_* + \int^x_0 (1+\OO(\eps)) |x'|^{\frac{1}{\beta}} \: dx' \\
& = y_* + ( \tfrac{\beta}{\beta+1} + \OO(\eps)) \sgn(x) |x|^{1+\frac{1}{\beta}}.
\end{align*}
It follows from this equation that $\sgn(y-y_*) = \sgn(x)$, and
\begin{equation}
\label{eq:x:1}
x = \sgn(y-y_*) \big( (1+ \tfrac{1}{\beta})^{\frac{\beta}{\beta+1}} + \OO(\eps) \big) |y-y_*|^{\frac{\beta}{\beta+1}}.
\end{equation}
Therefore, \eqref{eq:W_x:thm:1} gives us
\begin{align*}
w(y,T_*) & = W(x,T_*) \\
& = W(0,T_*) + W_x(0,T_*) x + \int^x_0 W_x(x',T_*)-W_x(0,T_*) \: dx' \\
& = w(y_*,T_*) + (-1+\OO(\eps)) x + \int^x_0 (1+\OO(\eps)) |x'|^{\frac{1}{\beta}} \: dx' \\
& = w(y_*,T_*) + (-1+\OO(\eps)) x + (1+\OO(\eps)) \tfrac{\beta}{\beta+1} \sgn(x) |x|^{1+\frac{1}{\beta}} \\
& = w(y_*, T_*) - \big( (1+\tfrac{1}{\beta})^{\frac{\beta}{\beta+1}} + \OO(\eps)\big) \sgn(y-y_*) |y-y_*|^{\frac{\beta}{\beta+1}} + (1+\OO(\eps)) (y-y_*).
\end{align*}
We have just proven that \eqref{eq:thm:expansion:1} holds for all $y \in \TT$ such that $|x| \leq 1$.

Let's now estimate $y_*$ and the radius of the neighborhood around $y_*$ for which $|x| \leq 1$. \eqref{eq:T_*} gives us
\begin{align*}
y_* 
= \eta(0,T_*)
& = \tfrac{1+\alpha}{2} \int^{T_*}_0 W(0,t) \: dt + \tfrac{1-\alpha}{2}\int^{T_*}_0 Z(0,t) \: dt \\
& = \tfrac{1+\alpha}{2} \int^{T_*}_0 w_0(0) + \OO(\eps)\: dt + \OO(\eps) \\
& = \tfrac{1+\alpha}{2} T_* (1+ \OO(\eps)) \\
& = 1 + \OO(\eps).
\end{align*}
We also compute that
\begin{align*}
\eta(1,T_*) - y_* & = \int^1_0 \eta_x(x,T_*) \: dt \\
& = \int^1_0 1 + \tfrac{1+\alpha}{2}T_* w'_0(x) + \OO(\eps) \: dx
& = 1 + (1+\OO(\eps)) \int^1_0 -1 + |x|^{\frac{1}{\beta}} \: dx + \OO(\eps) \\
& = \tfrac{\beta}{\beta+1} + \OO(\eps).
\end{align*}
Therefore, the neighborhood $\{ y \in \TT : x \in [0,1]\}$ corresponds to $\{ y> y_* : |y-y_*| \leq r \}$ for some $r = \frac{\beta}{\beta+1} + \OO(\eps)$. An analogous computation proves an analogous result for labels $x \in [-1,0]$.

Let us now get the H\"{o}lder regularity estimates for $\p_y z(\cdot, t)$ and $\p_y k(\cdot, t)$. Let $y = \eta(x,t)$. Since $\eta_x \leq 3+\alpha$ everywhere (see Lemma \ref{lem:1:1}), we know that
\begin{align}
|y_2-y_1| & = \bigg\vert \int^{x_2}_{x_1} \eta_x(x,t) \: dx \bigg\vert
\leq (3+\alpha) |x_2-x_1|
\end{align}
for all $x_1, x_2 \in \TT, t \in [0,T_*]$. We also know that $\eta_x \geq \frac{1}{2} + \OO(\eps)$ for $0 \leq t \leq \frac{1}{1+\alpha}$ and that $\eta_x > \frac{1}{2}$ for $1\leq |x| \leq \pi$. For $|x|\leq 1$ and $\frac{1}{1+\alpha} \leq t \leq T_*$ \eqref{eq:eta_x:min} gives us
\begin{equation*}
\eta_x \geq (\tfrac{1}{2} + \OO(\eps) ) |x|^{\frac{1}{\beta}}
\end{equation*}
for all $|x| \leq 1, \frac{1}{1+\alpha} \leq t \leq T_*$. It follows that
\begin{equation*}
\tfrac{1}{\eta_x} \leq \begin{cases} (2+ \OO(\eps))|x|^{-\frac{1}{\beta}} \hspace{10mm} |x| \leq 1 \text{ and } \frac{1}{1+\alpha} \leq t \leq T_* \\
2+ \OO(\eps) \hspace{25mm} \text{ otherwise} \end{cases} .
\end{equation*}
If we define $y_0 = \eta(0,t)$, then for all $ \frac{1}{1+\alpha} \leq t \leq T_*, \eta(-1,t) \leq y_1 \leq y_2 \leq \eta(1,t)$ we have
\begin{align*}
x_2-x_1 & = \int^{y_2}_{y_1} \frac{1}{\eta_x(x,t)} dy \\
& \leq (2+\OO(\eps)) \int^{y_2}_{y_1} |x|^{-\frac{1}{\beta}} \: dy \\
& \leq (2+\OO(\eps)) (3+\alpha)^{\frac{1}{\beta}} \int^{y_2}_{y_1} |y-y_0|^{-\frac{1}{\beta}} \: dy \\
& \leq (1+\tfrac{1}{\beta})(2+\OO(\eps)) (3+\alpha)^{\frac{1}{\beta}} \begin{cases} |y_2-y_0|^{\frac{\beta}{\beta+1}} + |y_0-y_1|^{\frac{\beta}{\beta+1}} \hspace{10mm} y_1 < y_0 \leq y_2 \\ \big\vert |y_2-y_0|^{\frac{\beta}{\beta+1}} - |y_1-y_0|^{\frac{\beta}{\beta+1}}\big\vert \hspace{8mm} \text{otherwise} \end{cases} \\
& \leq 2^{\frac{1}{\beta+1}}(1+\tfrac{1}{\beta})(2+\OO(\eps)) (3+\alpha)^{\frac{1}{\beta}} |y_2-y_1|^{\frac{\beta}{\beta+1}}.
\end{align*}
It is straightforward from here to verify that
\begin{equation}
\label{ineq:x:y}
|x_2-x_1| \lesssim |y_2-y_1|^{\frac{\beta}{\beta+1}} \hspace{10mm} \forall \: x_1,x_2 \in \TT, t \in [0,T_*].
\end{equation}

Since
\begin{align*}
\p_y z(y,t) & = \Z(x,t)-\tfrac{1}{2\gamma} \Sigma(x,t) \K(x,t), \\
\p_y k(y,t) & = \K(x,t),
\end{align*}
it now follows from \eqref{estimates:holder} and \eqref{ineq:x:y} that
\begin{equation*}
[\p_y z(\cdot, t) ]_{C^{0,\frac{1}{\beta+1}}} \lesssim \eps, \qquad [\p_y k(\cdot, t) ]_{C^{0,\frac{1}{\beta+1}}} \lesssim \eps.
\end{equation*}
This gives us the estimates \eqref{eq:thm:z:k} from Theorem \ref{thm:main}.
\qed

\subsection*{Acknowledgments} 
S.S.~was supported by  NSF grant DMS-2007606 and the Department of Energy Advanced Simulation and Computing (ASC) Program.  I.N.~and V.V.~were supported by the NSF CAREER grant DMS-1911413.  


\begin{bibdiv}
\begin{biblist}

\bib{BuDrShVi2021}{article}{
      author={Buckmaster, Tristan},
      author={Drivas, Theodore~D},
      author={Shkoller, Steve},
      author={Vicol, Vlad},
       title={{Simultaneous development of shocks and cusps for 2D Euler with
  azimuthal symmetry from smooth data}},
        date={2022},
     journal={Ann. PDE},
      volume={8},
      number={2},
       pages={pp.~199},
}

\bib{buckmaster2022formation}{article}{
      author={Buckmaster, Tristan},
      author={Iyer, Sameer},
       title={Formation of unstable shocks for 2d isentropic compressible
  euler},
        date={2022},
     journal={Communications in Mathematical Physics},
      volume={389},
      number={1},
       pages={197\ndash 271},
}

\bib{BuShVi2019b}{article}{
      author={Buckmaster, Tristan},
      author={Shkoller, Steve},
      author={Vicol, Vlad},
       title={{Formation of point shocks for 3D compressible Euler}},
        date={2022},
     journal={Communications on Pure and Applied Mathematics},
      eprint={https://onlinelibrary.wiley.com/doi/10.1002/cpa.22068},
         url={https://onlinelibrary.wiley.com/doi/10.1002/cpa.22068},
}

\bib{BuShVi2020}{article}{
      author={Buckmaster, Tristan},
      author={Shkoller, Steve},
      author={Vicol, Vlad},
       title={{Shock formation and vorticity creation for 3d Euler}},
        date={2022},
     journal={Communications on Pure and Applied Mathematics},
      eprint={https://onlinelibrary.wiley.com/doi/10.1002/cpa.22067},
         url={https://onlinelibrary.wiley.com/doi/10.1002/cpa.22067},
}

\bib{BuShVi2019a}{article}{
      author={Buckmaster, Tristan},
      author={Shkoller, Steve},
      author={Vicol, Vlad},
       title={Formation of shocks for 2d isentropic compressible euler},
        date={2022},
     journal={Communications on Pure and Applied Mathematics},
      volume={75},
      number={9},
       pages={2069\ndash 2120},
}

\bib{chen2002cauchy}{incollection}{
      author={Chen, Gui-Qiang},
      author={Wang, Dehua},
       title={The cauchy problem for the euler equations for compressible
  fluids},
        date={2002},
   booktitle={Handbook of mathematical fluid dynamics},
      volume={1},
   publisher={Elsevier},
       pages={421\ndash 543},
}

\bib{chen2001formation}{article}{
      author={Chen, Shuxing},
      author={Dong, Liming},
       title={Formation and construction of shock for {p}-system},
        date={2001},
     journal={Science in China Series A: Mathematics},
      volume={44},
      number={9},
       pages={1139\ndash 1147},
}

\bib{christodoulou2007formation}{book}{
      author={Christodoulou, Demetrios},
       title={The formation of shocks in 3-dimensional fluids},
   publisher={European Mathematical Society},
        date={2007},
      volume={2},
}

\bib{collot2018singularityburgers}{article}{
      author={Collot, Charles},
      author={Ghoul, Tej-Eddine},
      author={Masmoudi, Nader},
       title={Singularity formation for burgers equation with transverse
  viscosity},
        date={2022},
     journal={Ann. Sci. {\'E}c. Norm. Sup{\'e}r. (4)},
      volume={55},
      number={4},
       pages={1047\ndash 1133},
}

\bib{dafermos2005hyperbolic}{book}{
      author={Dafermos, Constantine~M},
       title={Hyperbolic conservation laws in continuum physics},
   publisher={Springer},
        date={2005},
      volume={3},
}

\bib{eggers2008role}{article}{
      author={Eggers, Jens},
      author={Fontelos, Marco~A},
       title={The role of self-similarity in singularities of partial
  differential equations},
        date={2008},
     journal={Nonlinearity},
      volume={22},
      number={1},
       pages={R1},
         url={https://doi.org/10.1088/0951-7715/22/1/001},
}

\bib{eggers2015singularities}{book}{
      author={Eggers, Jens},
      author={Fontelos, Marco~Antonio},
       title={Singularities: formation, structure, and propagation},
   publisher={Cambridge University Press},
        date={2015},
      volume={53},
}

\bib{IyRiShVi2023}{article}{
      author={Iyer, Sameer},
      author={Rickard, Calum},
      author={Shkoller, Steve},
      author={Vicol, Vlad},
       title={A detailed description of all $c^{\frac{1}{2k+1}}$ shock
  structures for 2d compressible euler in azimuthal symmetry},
        date={2023},
     journal={preprint},
}

\bib{John74}{article}{
      author={John, Fritz},
       title={Formation of singularities in one-dimensional nonlinear wave
  propagation},
        date={1974},
        ISSN={0010-3640},
     journal={Comm. Pure Appl. Math.},
      volume={27},
       pages={377\ndash 405},
         url={https://doi.org/10.1002/cpa.3160270307},
      review={\MR{369934}},
}

\bib{Kong2002}{article}{
      author={Kong, D.-X.},
       title={Formation and propagation of singularities for {$2\times2$}
  quasilinear hyperbolic systems},
        date={2002},
     journal={Transactions of the American Mathematical Society},
      volume={354},
      number={8},
       pages={3155\ndash 3179},
  url={https://www.ams.org/journals/tran/2002-354-08/S0002-9947-02-02982-3/},
}

\bib{lax1964development}{article}{
      author={Lax, Peter~D},
       title={Development of singularities of solutions of nonlinear hyperbolic
  partial differential equations},
        date={1964},
     journal={Journal of Mathematical Physics},
      volume={5},
      number={5},
       pages={611\ndash 613},
}

\bib{lebaud1994description}{article}{
      author={Lebaud, MP},
       title={Description de la formation d'un choc dans le p-syst{\`e}me},
        date={1994},
     journal={Journal de math{\'e}matiques pures et appliqu{\'e}es},
      volume={73},
      number={6},
       pages={523\ndash 566},
}

\bib{liu1979development}{article}{
      author={Liu, Tai-Ping},
       title={Development of singularities in the nonlinear waves for
  quasi-linear hyperbolic partial differential equations},
        date={1979},
     journal={Journal of Differential Equations},
      volume={33},
      number={1},
       pages={92\ndash 111},
         url={https://mathscinet.ams.org/mathscinet-getitem?mr=540819},
}

\bib{liu2021shock}{book}{
      author={Liu, Tai-Ping},
       title={Shock waves},
   publisher={American Mathematical Soc.},
        date={2021},
      volume={215},
}

\bib{luk2018shock}{article}{
      author={Luk, Jonathan},
      author={Speck, Jared},
       title={Shock formation in solutions to the 2d compressible euler
  equations in the presence of non-zero vorticity},
        date={2018},
     journal={Inventiones mathematicae},
      volume={214},
      number={1},
       pages={1\ndash 169},
         url={https://doi.org/10.1007/s00222-018-0799-8},
}

\bib{luk2021stability}{article}{
      author={Luk, Jonathan},
      author={Speck, Jared},
       title={The stability of simple plane-symmetric shock formation for 3d
  compressible euler flow with vorticity and entropy},
        date={2021},
     journal={arXiv preprint arXiv:2107.03426},
}

\bib{majda1984compressible}{book}{
      author={Majda, Andrew},
       title={Compressible fluid flow and systems of conservation laws in
  several space variables},
   publisher={Springer-Verlag},
        date={1984},
      volume={53},
         url={https://doi.org/10.1007/978-1-4612-1116-7},
}

\bib{neal2023}{article}{
      author={Neal, Isaac},
      author={Shkoller, Steve},
      author={Vicol, Vlad},
       title={A characteristics approach to shock formation in {2D} {E}uler
  with azimuthal symmetry and entropy},
        date={202302},
     journal={arXiv preprint arXiv:2302.01289},
      eprint={2302.01289},
         url={https://arxiv.org/pdf/2302.01289.pdf},
}

\bib{sideris1985formation}{article}{
      author={Sideris, Thomas~C},
       title={Formation of singularities in three-dimensional compressible
  fluids},
        date={1985},
     journal={Communications in mathematical physics},
      volume={101},
      number={4},
       pages={475\ndash 485},
         url={http://projecteuclid.org/euclid.cmp/1104114244},
}

\bib{yin2004formation}{article}{
      author={Yin, Huicheng},
       title={Formation and construction of a shock wave for 3-d compressible
  euler equations with the spherical initial data},
        date={2004},
     journal={Nagoya Mathematical Journal},
      volume={175},
       pages={125\ndash 164},
         url={https://doi.org/10.1017/S002776300000893X},
}

\end{biblist}
\end{bibdiv}


\end{document}